\documentclass[11pt]{article}
\usepackage{latexsym,amsmath,amsthm,verbatim,ifthen,amssymb,graphicx,color}
\usepackage[all]{xy}

\SelectTips{cm}{}

\newtheorem{theorem}{Theorem}[section]

\newtheorem{corollary}[theorem]{Corollary}
 \newtheorem{lemma}[theorem]{Lemma}
 \newtheorem{proposition}[theorem]{Proposition}
 \theoremstyle{definition}
 \newtheorem{definition}[theorem]{Definition}
 \theoremstyle{remark}
 \newtheorem{remark}[theorem]{Remark}
 \newtheorem{example}[theorem]{Example}
 \numberwithin{equation}{subsection}

\newcommand{\bz}{\mathbb Z}
\newcommand{\bk}{\mathbb K}
\newcommand{\bq}{\mathbb Q}
\newcommand{\sgn }{\text{sgn}}

    \newcommand{\lasu}{{\mathfrak{L}}}
\newcommand{\ad }{\text{ad}}
\newcommand{\im}{\text{Im}\,}

 \newcommand{\dalg}{d}
 
 \newcommand{\dlie}{\partial }

 \newcommand{\lib }{\mathbb{L}}

 \newcommand{\cyl}{\operatorname{\textrm{Cyl}\,}}
 
 \newcommand{\cinfinito}{\operatorname{{\mathcal C}^\infty}}
  
  \newcommand{\csubinfinito}{\operatorname{{\mathcal C}_\infty}}
\newcommand{\catcdga}{\operatorname{{\bf CDGA}}}
\newcommand{\catcono}{\operatorname{{\bf DGC_\infty}}}
\newcommand{\catco}{\operatorname{{\bf CDGC_\infty}}}

\newcommand{\catdga}{\operatorname{{\bf DGA}}}
\newcommand{\catdgl}{\operatorname{{\bf DGL}}}
\newcommand{\cobarinf}{\operatorname{{{\rm cobar}_\infty}}}
\newcommand{\catset}{\operatorname{{\bf Set}}}
\newcommand{\catsimpset}{\operatorname{{\bf SSet}}}
\newcommand{\catlinfinito}{\operatorname{{\bf L_\infty}}}
\newcommand{\aug}{\operatorname{\textit{Aug}\,}}

 \newcommand{\pl }{1}

 \newcommand{\lasumuno}{ \mathfrak{M}^{a_0}}
 \newcommand{\lasumdos}{ \mathfrak{M}^{a_1}}
 \newcommand{\lasum}{ \mathfrak{M}}
 \newcommand{\partialuno}{ \partial^{a_0}}

 \newcommand{\partialdos}{ \partial^{a_1}}

 \newcommand{\MC}{\operatorname{{\rm MC}}}
\newcommand{\mc}{{\MC}}

\newcommand{\basum}{ \mathfrak{B}}

\begin{document}
\title{Algebraic models of non-connected spaces and homotopy theory of $L_\infty$ algebras}


\author{Urtzi Buijs\footnote{Partially supported by the \emph{Ministerio de Educaci\'on y
Ciencia} grant MTM2010-15831, by the  grants FQM-213, 2009-SGR-119, and by the U-mobility grant 246550 of the European Union Seventh Framework Program.} $$ and Aniceto Murillo\footnote{Partially supported by the \emph{Ministerio de Educaci\'on y
Ciencia}  grant MTM2010-18089 and by the Junta de
Andaluc\'\i a grants FQM-213 and P07-FQM-2863. \vskip 1pt 2000 Mathematics Subject
Classification: 55P62.\vskip
 1pt
 Key words and phrases: Rational homotopy theory. Algebraic models of non-connected spaces. $L_\infty$-algebras. Maurer-Cartan set.}}




\maketitle

\begin{abstract}
We develop a homotopy theory of $L_\infty$ algebras based on the Law-rence-Sullivan construction, a complete differential graded Lie algebra which, as we show, satisfies the necessary properties  to become the right cylinder in this category. As a result, we obtain a general procedure to algebraically model the rational homotopy type of non-connected spaces.
\end{abstract}

\section*{Introduction}
It has been known for a long time, probably since the work of Gerstenhaber \cite{gers} in the first case and that of Quillen \cite{qui} in the second, that deformation functors on associative algebras and rational homotopy types of spaces  are governed by differential graded Lie algebras together with solutions of the classical Maurer-Cartan equation modulo gauge equivalence. Also, in this two contexts, $L_\infty$ algebras are proven to be the right objects to attack certain problems in which the rigidity of classical differential graded Lie algebras was an obstacle.

However, the available closed model structures on the categories of unbounded differential graded Lie algebras or $L_\infty$ algebras no longer fully reflect the homotopy properties of their realizations.

In this paper, having as goal shaping algebraically the rational homotopy type of non-connected spaces, we develop a precise and functorial  homotopy theory of $L_\infty$ algebras based   on a particular complete differential graded Lie algebra, namely the {\it Lawrence-Sullivan construction} \cite{lasu}.

 To this end we prove the following, which may also be of independent interest: first, we ``approximate'' the linear dual  of the standard acyclic algebra, model of a point, $\Lambda(t,dt)$, by a cocommutative differential graded coalgebra ${\mathfrak B}$; see Definition \ref{universal2}. Next,  we consider a simple acyclic differential vector space   $V=\langle y,z,c\rangle$ with $c$ of degree one, $y$ and $z$ of degree zero and $dc=y-z$. Then, we prove that $(V,d)$ is a deformation retract of $({\mathfrak B},\delta)$
$$
\xymatrix{ V\ar@{^{(}->}[r]& {\mathfrak B}\ar@(ur,dr) \ar@<0.75ex>[l] }
$$
 in such a way that the inherited $A_\infty$ coalgebra structure on $V$, via the {\em classical perturbation theorem}  known nowadays as the {\em homotopy transfer theorem}, or equivalently, the inherited differential graded algebra structure on $\widehat T(s^{-1}V)$ is precisely the universal enveloping algebra of the Lawrence-Sullivan construction (see Theorem \ref{principal2} for a precise statement).

 This is crucial to show later on that the Lawrence-Sullivan construction is in fact   the right cylinder  to develop  the homotopy relation in Maurer-Cartan elements of $L_\infty$ algebras. We immediately prove afterwards that  this notion coincides with the several equivalent ones, broadly used by experts nowadays,  both in  rational homotopy and deformation theory.


    As we said, we are able to use this detailed theory to present a procedure to  construct  algebraic models which fully describe the rational homotopy of non-connected spaces. Starting with an arbitrary  family of nilpotent spaces of finite type and some choice for their algebraic models, we are able to glue them together and obtain a differential graded Lie algebra whose classical geometric realization via simplicial cochains, is precisely the rationalization of a non-connected space having the original family as components (see Theorem \ref{modelo}).

    As a result of this we see, for instance, that, as expected from its functorial properties and  its fundamental role as a cylinder of this theory, the Lawrence-Sullivan construction is a model of the disjoint union of the interval and an external point. This procedure may also reveal, at least up to homotopy, the geometry reflected by the generalized Quillen construction of the $A_\infty$ coalgebra structure on the chains of a cell complex given in  \cite[Appendix]{tradzeisu} or \cite{lasu}.

      The paper is organized as follows. The first section contains the fundamental definitions and basic facts  on $L_\infty$ algebras and their geometric realization functor. In Section $2$ we  present in detail, and with the necessary restrictions, the one-to-one correspondence between the Maurer-Cartan set of a given $L_\infty$ algebra $L$ with  based maps from the zero sphere $S^0$ into the realization of $L$, that is, with points of this space. We also show that Maurer-Cartan elements are preserved in the standard way, and it becomes a functor to the category of sets, only with the appropriate restrictions.
      In Section $3$ we introduce the Lawrence-Sullivan construction, present its functorial properties and prove that it can be obtained as an inherited infinity structure from the standard acyclic coalgebra. We use this construction as a cylinder to introduce the homotopy relation of Maurer-Cartan elements of a given $L_\infty$ algebra in Section $4$. We see that this is the natural link to unify different approaches to homotopy in this setting. Section $5$ is devoted to localize componentwise, both algebraically and geometrically, a given $L_\infty$ algebra, and to precise the homotopy invariance of this procedure. Finally, in Section $6$ we present the mentioned algorithm to obtain models of the rational homotopy type of non-connected spaces.

      We finish by remarking that, throughout this article,  commonly believed and widely used extensions of classical geometrical properties of Lie algebras to $L_\infty$ algebras are  true only when some restrictions (quite severe in some cases) are imposed. For instance, and independently from finiteness and or bounding assumptions, ``mildness conditions" (see next section for a precise definition) on  $L_\infty$ algebras and their morphisms are required to functorially  define their cochain algebras, even if one wants to consider complete cochains or limit of cochains in a good chosen filtration of the $L_\infty$ algebra. Also, to build a consistent theory, the set of Maurer-Cartan elements should be preserved by morphisms in the standard way. However, we see that this is the case only when technical finiteness restrictions are assumed.  The same kind of considerations and restrictions apply to the broadly extended principle by which Maurer-Cartan elements of a given $L_\infty$ algebra $L$ are in bijective correspondence with the augmentations of the cochain algebra of $L$.

\medskip

\noindent{\bf Acknowledgements.} We thank the referee, not only for for his/her priceless suggestions and corrections which have substantially improved both the content and the presentation of the paper, but also, for his/her  endless patience with the authors during the review process. We also thank Prof. Jim Stasheff and Prof. Vladimir Dotsenko for  his kind comments and his valuable  suggestions on the first manuscript of this work.

\section{Algebraic models of $L_\infty$ structures}

With the aim of fixing
notation, we give in this section some definitions and sketch some results we will
need on  $L_\infty$ structures, particularly the ones coming from their interaction with rational homotopy theory. For the latter, \cite{fehatho} is a standard reference while a good geometric and homotopy oriented introduction to $L_\infty$ algebras can be found in  \cite{getz,kon}.

We begin by setting some general assumptions. Abusing notation, we will denote  in the same way a given category, always written in bold face, and the class of its objects. The coefficient field for any algebraic object $\bk$ is assumed to be of characteristic zero. Any graded object is considered $\bz$-graded unless explicitly specified otherwise. The degree of a homogeneous element $x$ in such an object will be denoted by $|x|$.  No finite type assumptions will be made throughout the paper. Thus, whenever a basis $\{v_i\}_{i\in I}$ of a homogeneous vector space $V$ is fixed, a  vector of the dual $V^\sharp$ will be written as a formal series $\sum_{i\in I}\lambda_iv_i^\sharp$ representing the map $\{v_i\}_{i\in I}\to\bk$ which sends each $v_i$ to $\lambda_i$.

Our graded coalgebras are coassociative, generally cocommutative, but are not assumed to have a counit neither to be coaugmented. It is known however, that the functor $C\mapsto \ker \varepsilon$ which assigns to each coaugmented graded coalgebra the kernel of the coaugmentation $\varepsilon\colon C\to\bk$ provides an equivalence between the categories of coaugmented graded coalgebras and graded coalgebras. The inverse functor sends $(C,\Delta)$ to the augmented graded coalgebra $(\bk u\oplus C,\Delta')$ in which $u$ is a degree zero element, $\Delta'u=u\otimes u$ and $\Delta'c=u\otimes c+c\otimes u+\Delta c$. We will use this equivalence whenever we need to coaugment any given graded coalgebra.
On the other hand, our graded algebras are assumed to have a unit which is preserved by morphisms.

       An {\em $L_\infty$ algebra} structure on a graded vector space $L$, denoted sometimes by $(L, \{\ell_k\}_{k\ge 1})$, is a collection of linear maps, called brackets, $\ell_k$, $k\ge 1$, of degree $k-2$
       $$
       \ell_k=[\,,...\,,\,]\colon \otimes^kL\to L,
       $$
 which satisfy:

       (1) $\ell_k$ are graded skew-symmetric, i.e., for any $k$-permutation $\sigma$,
       $$
       [x_{\sigma(1)},...\,,x_{\sigma (k)}]=\sgn(\sigma)\varepsilon_\sigma[x_1,...\,,x_k],
       $$
       where $\varepsilon_\sigma$ is the sign given by the Koszul convention.

       (2) The following generalized Jacobi identities hold:
       $$
       \sum_{i=1}^{n}\sum_{\sigma\in S(i,n-i)}\sgn(\sigma)\varepsilon_\sigma(-1)^{i(n-i)}\bigl[[x_{\sigma(1)},...\,,x_{\sigma(i)}],x_{\sigma(i+1)},...\,,x_{\sigma(n)}\bigr]=0.
       $$
       By $S(i,n-i)$ we denote the $(i,n-i)$ {\em shuffles} whose elements are permutations $\sigma$ such that $\sigma(1)<...<\sigma(i)$ and
$\sigma(i+1)<...<\sigma(n)$.

 Recall that $L_\infty$ structures in $L$ are in one-to-one correspondence with codifferentials on the non coaugmented cofree cocommutative coalgebra $\Lambda^+ sL$ cogenerated by $sL$, in which $s$ denotes suspension, i.e., $(sL)_k=L_{k-1}$. Indeed, a codifferential $\delta$ on $\Lambda^+ sL$ is determined by a degree
$-1$ linear map $ \Lambda^+ sL\to sL$ which is written as the sum of
linear maps $\delta^{(k)}\colon\Lambda^ksL\to sL$, $k\ge 1$. Then the operators $\{\ell_k\}_{k\ge 1}$ on $L$ and the codifferential $\delta$ on $\Lambda sL$ are uniquely determined by each other via
$$
\begin{aligned}
\ell_k&=s^{-1}\circ \delta^{(k)}\circ s^{\otimes k}\colon \otimes^k L\to
L,\\
\delta^{(k)}&=(-1)^{\frac{k(k-1)}{2}}s\circ\ell_k\circ(s^{-1})^{\otimes k}\colon\Lambda^ksL\to sL.
\end{aligned}
$$
As  $\ell_1$ is simply a differential on $L$, we often refer to it as $\partial$.

Note that, via the equivalence above, we may identify the cocommutative differential graded coalgebra, CDGC henceforth, $(\Lambda^+ sL, \delta)$ with the  CDGC $(\Lambda sL,\delta)$, naturally coaugmented by the new $\delta$-cycle $1\in \Lambda sL$.
In what follows, we denote by $\mathcal{C}_\infty (L)=(\Lambda sL, \delta)$ this coaugmented CDGC corresponding to the $L_\infty$ structure on $L$.

       Given two $L_\infty$ algebras $L$ and $L'$, a {\em morphism of $L_\infty$
algebras}  or an {\em $L_\infty$ morphism}  is a CDGC morphism,
$$f\colon \mathcal{C}_\infty (L)=(\Lambda sL,\delta)\to(\Lambda sL',\delta')=\mathcal{C}_\infty (L').
$$
Observe that $f$ is determined by  $\pi
f\colon \Lambda sL\to sL'$ ($\pi$ denotes the projection) which can be written as
$ \sum_{k\ge 1}(\pi f)^{(k)}$, where $(\pi f)^{(k)}\colon \Lambda^ksL\to sL'$. Note that, as before, the collection of linear maps $\{(\pi f)^{(k)}\}_{k\ge 1}$ is in one to one correspondence with a system $\{f^{(k)}\}_{k\ge1}$ of
skew-symmetric maps of degree $1-k$, where $f^{(k)}\colon \otimes^k L\to L'$. Indeed each $f^{(k)}$ is uniquely determined by $(\pi f)^{(k)}$ as follows:
$$
\begin{aligned}
f^{(k)}&=s^{-1}\circ (\pi f)^{(k)}\circ s^{\otimes k},\\
(\pi f)^{(k)}&=(-1)^{\frac{k(k-1)}{2}}s\circ f^{(k)}\circ (s^{-1})^{\otimes k}.\\
\end{aligned}
$$
Observe that, as in the relation between $\ell_k$ and $\delta^{(k)}$, the sign $(-1)^{\frac{k(k-1)}{2}}$ is needed to have the above equalities.

The fact that $f$ is a CDGC morphism is equivalent to state that the system $\{f^{(k)}\}_{k\ge1}$ satisfies  an infinite sequence of equations
involving the brackets $\ell_k$ and $\ell'_k$, $k\ge 1$:

\medskip

              $\ell'_1f^{(1)}=f^{(1)}\ell_1$, i.e., $f^{(1)}\colon (L,\ell_1)\to (L',\ell'_1)$ is a differential map,\hfill\break

              $\ell'_1\bigl(f^{(2)}(x\otimes y)\bigr) +\ell'_2\bigl(f^{(1)}(x)\otimes f^{(1)}(y)\bigr)=$\hfill\break
                \indent$ =f^{(1)}\bigl(\ell_2(x\otimes y)\bigr)+f^{(2)}\bigl(\ell_1( x)\otimes y-(-1)^{|x|}x\otimes\ell_1( y)\bigr)$,\hfill${(1)}$\break

              $\cdots\cdots\cdots$

\medskip

Abusing notation, and whenever there is no ambiguity, we shall often denote an $L_\infty$ morphism simply by $f\colon L\to L'$.  An $L_\infty$ morphism $f$ is a {\em quasi-isomorphism} if $f^{(1)}\colon (L,\ell_1)\stackrel{\simeq}{\to}(L',\ell'_1)$ is a  quasi-isomorphism of differential graded vector spaces, or equivalently, if  $f\colon \mathcal{C}_\infty(L)\stackrel{\simeq}{\to}\mathcal{C}_\infty(L')$ is a CDGC quasi-isomorphism.

\begin{definition}
As we do not want to restrict to  $L_\infty$ algebras which are nilpotent or nilpotently filtered, see Remark  \ref{comprendo} below,
a {\em Maurer-Cartan element} of an  $L_\infty$ algebra is an element $z\in L_{-1}$ for which $\ell_k(z,\stackrel{k}{\ldots}\,,
z)=0$ for $k$ sufficiently large and
$$
\sum_{k=1}^{\infty}\frac{1}{k!}\ell_k(z,\stackrel{k}{\ldots}\,,
z)=\sum_{k =1}^\infty \frac{1}{k !}[ z ^{\wedge  k } ]=0.
$$
Observe that, whenever $(L,\partial)$ is a differential Lie algebra, DGL henceforth, i.e., an $L_\infty$ algebra such that $\ell_k=0$ for $k\ge 3$, then $z\in L_{-1}$ is a Maurer-Cartan element if
$$
\partial z=-\frac{1}{2}[z,z].
$$
We will denote the set of Maurer-Cartan elements in $L$ by ${\rm
MC}(L)$.
\end{definition}

Given an $L_{\infty}$ algebra $L$ and $z\in \mc(L)$, define the {\em perturbation of $\ell_k$ by $z$} as
$$
\ell_k^z(x_1,\ldots,
x_k)=[x_1,\dots ,x_k]_z=\sum_{i=0}^{\infty}\frac{1}{i!}\ell_{i+k}(z,\stackrel{i}{\dots}\,,z,x_1,\ldots,
x_k).
$$
Whenever the above sum is always finite, $(L,\{\ell_i^z\})$ is again an $L_\infty $ algebra ~\cite[Proposition~4.4]{getz} which will be denoted by $L^z$.

Given an $L_\infty$ algebra $L$ and a commutative differential graded algebra $A$, CDGA from now on, the tensor product $L\otimes A$ inherits a natural $L_\infty$ structure with brackets
$$\begin{aligned}
&\ell_1(x\otimes a)=\partial x\otimes a +(-1)^{|x|}x\otimes da,\\
&\ell_k(x_1\otimes a_1,\ldots,x_k\otimes a_k)=\varepsilon \ell_k(x_1,\ldots,x_k)\otimes a_1\ldots a_k,\qquad k\ge 2,
\end{aligned}
$$
where $\varepsilon=(-1)^{\sum_{i>j}|x_i||a_j|}$ is the sign provided by the Koszul convention.

Realization of $L_\infty$ algebras is done via the cochain functor. Hence, with the only purpose of unambiguously defining this functor, we need to constrain the class of $L_\infty$ algebras we work with.

\begin{definition}\label{mild} An $L_\infty$ algebra $L$ is {\em mild} if every bracket is locally finite, i.e., for any $a\in L$ there are finite dimensional subspaces $S_k\subset \otimes^kL$, $k\ge 1$,  which vanish for $k$ sufficiently large, and such that
$$
\ell_k^{-1}\langle a\rangle\subset \ker\ell_k\oplus S_k.
$$
\end{definition}

Here, $\langle a\rangle$ denotes the linear subspace spanned by $a$. Note  that
 if $L$ is mild and $z\in\mc(L)$, then $L^z$ is also mild.

\medskip

\begin{remark}\label{finitud}
Observe that, even finite type  and bounded  $L_\infty$ algebras are not necessarily mild. Indeed, the $L_\infty$ algebra $L=L_{-1}\oplus L_{-2}$ with $L_{-1}=\langle a\rangle$, $L_{-2}=\langle b\rangle$ and $\ell_k(a,\dots,a)=b$ for any $k\ge 1$, is clearly not mild.  Also, mild $L_\infty$ algebras do not satisfy in general  finiteness or bounding properties. For instance, any non finite type, non bounded abelian $L_\infty$ algebra ($\ell_k=0$ for all $k\ge 1$) is trivially mild.
\end{remark}

\begin{definition}\label{cocadena}
Given $L$ a mild $L_\infty$ algebra, choose a homogeneous basis $\{z_i\}$ of $L$ and denote by $V\subset (sL)^\sharp$ the subspace of the dual of $sL$ generated by the forms $\{v_i\}$ in which $v_i(sz_r)=\delta_i^r$. We then define the commutative differential graded algebra $ \cinfinito(L)$ of {\em cochains} on $L$ as the  free commutative algebra $\Lambda V$ endowed with the differential $\dalg =\sum_{k \geq \pl } \dalg _k$ with $d_kV\subset\Lambda^kV$, naturally induced by the $L_\infty$ structure on $L$ via the pairing above:
$$\langle \dalg _kv;sx_1\wedge \dots \wedge sx_k \rangle
=\varepsilon \langle v;s\ell_k( x_1,\dots, x_k) \rangle,\eqno (2)$$
where $\varepsilon=(-1)^{|v|+\sum_{j=1}^{k-1}(k-j)|x_j|}$.
 This sign arises by thinking of $d_k$ as the  dual of $\delta^{(k)}$,  and writing $\delta^{(k)}$ in terms of the corresponding $\ell_k$.

 Note that this construction is independent of the choice of a basis because $V$ is a subspace of $(sL)^\sharp$ consisting of linear forms which vanish outside a finite dimensional subspace  of $sL$.

 Observe that mildness is essential for $d$ to be well defined. Indeed, the condition $\ell_k^{-1}\langle a\rangle\subset \ker\ell_k\oplus S_k$ is equivalent to the  fact that, for any $k\ge 1$ and any generator $v\in V$,   $d_kv$  is a finite sum. The condition that $S_k=0$ (or equivalently $\ell_k^{-1}\langle a\rangle\subset\ker \ell_k$) for $k$ sufficiently large,  translates to the fact that, again for any generator $v\in V$, $d_kv=0$ for $k$ sufficiently large.

 \end{definition}

Observe that, for a mild $L_\infty$ algebra  $L$ of finite type, $\cinfinito(L)=(\Lambda (sL)^\sharp,d)$ with the differential defined by equation $(2)$.
 Note also that, even if  $L$ is of finite type and bounded,  it is essential to assume  mildness  so that, for each $v\in V$, $d_kv$ is well defined and eventually null.  Observe also that,  whenever $(L,\partial)$ is a finite type, positively graded DGL, then $\cinfinito(L)$ is
 the classical cochain algebra on $(L,\partial)$; see \cite[\S23]{fehatho}.

The construction $\cinfinito(-)$ does not define a functor unless we also restrict the class of $L_\infty$ morphisms.

\begin{definition}\label{mildmorfismo} An $L_\infty$ morphism $\phi\colon (\Lambda sL,d)\to (\Lambda sM,d)$  is {\em mild} if every $\phi^{(k)}$ is locally finite, i.e., if for any $a\in M$ there are finite dimensional subspaces $S_k\subset \otimes^k L$,  $k\ge 1$,  which vanish for $k$ sufficiently large, and such that
${\phi^{(k)}}^{-1}\langle a\rangle\subset \ker \phi^{(k)}\oplus S_k$.
\end{definition}

\begin{definition}\label{cinfinitomorfismo} Given a mild $L_\infty$ morphism $\phi\colon (\Lambda sL,d)\to (\Lambda sM,d)$, define
$$
\cinfinito(\phi)\colon \cinfinito(M)\to\cinfinito (L)
$$
 as the following CDGA morphism: write  $\cinfinito(L)=\Lambda V$, $\cinfinito(M)=\Lambda W$. Then, on $W$ , $\mathcal{C}^\infty(g) =\sum_{k\geq 1}\mathcal{C}^\infty (g)_k$, with each $\mathcal{C}^\infty (g)_k\colon W\to \Lambda ^kV$ given  via the pairing,
$$\langle \mathcal{C}^\infty (g)_k w; sx_{1}\wedge \cdots \wedge sx_{k}\rangle =\varepsilon \langle w ; sg^{(k)}(x_{1}\otimes \cdots \otimes x_{k})\rangle\eqno{(3)}$$
where $\varepsilon=(-1)^{\sum _{j=1}^k(k-j)|x_{j}|}$. Again, this sign comes from setting first $\mathcal{C}^\infty (g)_k$ as the induced by the dual of $(\pi g)^{(k)}$ and then, writing this map in terms of $g^{(k)}$.
\end{definition}

 Observe that the  assumption ${\phi^{(k)}}^{-1}\langle a\rangle\subset \ker \phi^{(k)}\oplus S_k$  is equivalent to say that, for each $w\in W$, $\cinfinito(\phi)_k(w)$ is a finite sum, and thus it is well defined. On the other hand the  requirement $S_k=0$ (or equivalently, ${\phi^{(k)}}^{-1}\langle a\rangle\subset \ker \phi^{(k)}$) for $k$ sufficiently large, translates to the fact that, for each $w\in W$, $\cinfinito(\phi)_k(w)=0$ for $k$ sufficiently large, and thus $\cinfinito(\phi)$ is well defined.

 Note also that an $L_\infty$ morphism between mild $L_\infty$ algebras is not necessarily mild.

  From now on we denote by $\catlinfinito$ the category of mild $L_\infty$ algebras and mild morphisms.
Even imposing finiteness and bounding assumptions,  the following remark shows that $\catlinfinito$ is the most general category in which the free cochain functor can be properly defined. We also see that mildness and nilpotency are conceptually different notions.

\begin{remark}\label{comprendo}  (i) Observe that the CDGA $\csubinfinito(L)^{\sharp}$, dual of $(\Lambda sL,d)$,  is not isomorphic to $\Lambda(sL)^\sharp$ unless very strict restrictions are assumed. Indeed, if $U$ is a graded vector space   of finite type,  bounded (below or above) and $U_0=0$,  the proof in Proposition 23.1 of \cite{fehatho} shows that  $(\Lambda U)^\sharp\cong \Lambda U^\sharp$. However,  for instance, if $U$ is of finite type but $U_0$ is non-trivial, $\Lambda U_0$ is of infinite countable dimension and so is  $(\Lambda U)_p=\Lambda U_0\cdot(\Lambda U)_p$ for each $p\in\bz$. Hence, its dual $(\Lambda U)^{\sharp p}$ is of infinite uncountable dimension in contrast to the countable dimension of $(\Lambda U^\sharp)^p$. In general, the linear (non differential) projection $\Lambda sL\to sL$ induces an injection $(sL)^\sharp\hookrightarrow (\Lambda sL)^\sharp$ which extends to a  (non differential) commutative algebra  injection $\Lambda (sL)^\sharp\hookrightarrow\csubinfinito(L)^\sharp$. However, note that even for mild $L_\infty$ algebras, equation  $(2)$ above does not define a CDGA structure in $\Lambda (sL)^\sharp$.

Even if one wishes to think of the CDGA $\csubinfinito(L)^{\sharp}$ as a ``free" algebra, the ``generators'' are in general infinite series with possibly uncountable terms. This entails serious convergence problems which can only be solved by, again, imposing strong restrictions of several kinds in the coefficient field, the $L_\infty$ structure of $L$ or in the dual functor.

(ii)  It is important to note that mildness and the beautiful notion of nilpotency in $L_\infty$ algebras \cite[Definition 4.2]{getz} are different.  The first is assumed here to have a realization functor via cochains while the latter was imposed to realize the Maurer-Cartan set as a Kan complex \cite[\S4]{getz} homotopy equivalent to its nerve \cite[Corollary 5.9]{getz}. For instance, consider $(\Lambda V,d)$ the minimal Sullivan model of a wedge of two odd spheres of dimension greater than 1. Note that $V$ is of finite type, although infinite dimensional, concentrated in odd degrees, and the differential is quadratic. Moreover, there are elements $u,v_n\in V$, with $n\ge 1$ such that $dv_{n+1}=uv_{n}$. Trivially, $(\Lambda V,d)=\cinfinito(L)$ for the DGL $(L,\partial)$, in which $(sL)^\sharp=V$ and the differential and bracket are induced by $d$ and the product on $\Lambda V$ respectively. Thus, there are elements in $a,b_n\in L$, $n\ge 0$ for which $[a,b_n]=b_{n+1}$ and therefore $L$ is not nilpotent while it is obviously mild as it is of finite type and bounded below. On the other hand, and starting with a free DGL with zero differential on suitable generators, one may easily impose relations so that it becomes a non mild and nilpotent DGL.

 \end{remark}

\begin{definition}
The {\em realization} functor,
$$
 \langle-\rangle\colon \catlinfinito\longrightarrow \catsimpset.
 $$
 is the composition of
 $$
 \cinfinito\colon \catlinfinito\to \catcdga
 $$
 with the
 Sullivan  realization functor, denoted in the same way by
 $$
 \langle - \rangle\colon \catcdga\to \catsimpset.
 $$
 Recall, see for instance \cite[\S17]{fehatho}, that given $A$ a CDGA, $\langle A\rangle$ is the simplicial set in which $\langle A\rangle_n=Hom(A,(A_{PL})_n)$, where $A_{PL}$ is the simplicial CDGA given by the differential polynomial forms on the standard simplices, i.e.,
 $$(A_{PL})_n=\Lambda ( t_1,\ldots,t_n,dt_1,\ldots,dt_n)
 $$
  with $|t_i|=0$ for all $i$.

Finally, we say that an $L_\infty$ algebra $L$ is a {\em model} of a space or simplicial set $X$ if $\langle\cinfinito(L)\rangle$ has the homotopy type of $X$.
\end{definition}

\section{The Maurer-Cartan functor and CDGA morphisms}\label{maurer}

 We explicitly extend \cite[Remark 16]{bufemu} or \cite[Proposition 1.1]{getz}, whose germ already appears in \cite[Appendix B6]{qui} to the category $\catlinfinito$. To our knowledge, this extension is by no means straightforward.

Let $(\lib(u),\partial)$ be the free Lie algebra generated by the Maurer-Cartan element $u$, that is, $\partial u=-\frac{1}{2}[u,u]$. This DGL will play a fundamental role in what follows. Its cochain algebra is easily computed to yield
$$\cinfinito(\lib(u),\partial)\cong (\Lambda (x,y),d)$$
 where
 $x$ and
$y$ are generators of degrees $0$ and $-1$ respectively,  $dx=0$  and $dy=\frac{1}{2}(x^2-x)$. In fact, as graded vector space $\lib(u)=\langle u,[u,u]\rangle$. Thus, see Definition \ref{cocadena},
$$
\cinfinito(\lib(u),\partial)=\bigl(\Lambda\bigl(s\lib(u)\bigr)^\sharp,d\bigr)=(\Lambda (x,y),d)
$$
with $x=(su)^\sharp$, $y=(s[u,u])^\sharp$ of degrees $0$ and $-1$ respectively. If we compute the differential $d$ as in Definition  \ref{cocadena},  the only non vanishing terms of equation $(2)$ are the following:
$$
\begin{aligned}
&\langle d_1(s[u,u])^\sharp;su\rangle=-\langle(s[u,u])^\sharp;s\partial u\rangle=\frac{1}{2}\langle(s[u,u])^\sharp;s[u,u]\rangle=\frac{1}{2},\\
&\langle d_2(s[u,u])^\sharp;su\wedge su\rangle=\langle (s[u,u])^\sharp;s[u,u]\rangle=1.\\
\end{aligned}
$$
Hence, $dx=0$. On the other hand, from the first equation we get that $d_1y=\frac{1}{2}x$. Finally, from the second equation, together with the obvious equality $\langle x^2;su\wedge su\rangle=2$ we get that $d_2y=\frac{1}{2}x^2$. However, for practical and technical purposes in what follows, we will replace $x$ by $-x$ so that $dy=\frac{1}{2}(x^2-x)$.

It is   easy to check that the geometrical realization of $ (\Lambda (x,y),d)$ has the homotopy type of $S^0$. Moreover,  consider as in \cite{femutan} the CDGA model of $S^0$ given  by $\bq \alpha\oplus\bq\beta$ with $\alpha$ and $\beta$ idempotent elements of degree $0$,  $\alpha^2=\alpha$, $\beta^2=\beta$, with $\alpha\beta=0$.  Note that the identity in this algebra is $\alpha+\beta$. Hence, replacing $\alpha$ by $x$ and $\beta$ by $1-x$,  this CDGA is isomorphic to $\bq x\oplus\bq (1-x)$ with $x^2=x$, i.e. $\Lambda x/(x^2-x)$, which is precisely the cohomology algebra of  $\cinfinito(\lib(u),\partial)=(\Lambda (x,y),d)$.

\medskip

 We begin by the following  fundamental auxiliary result.

 \begin{lemma}\label{existenciaMC}
Let
 $L$ be an $L_\infty $ algebra. Then, for any  $z\in L_{-1}$,  there exists a unique $L_\infty$ morphism $\phi \colon (\lib (u),\partial)\to L$ such that
$$\begin{aligned}
&\phi^{(1)}(u)=z,\\
&\phi^{(k)}(u\otimes\ldots\otimes u)=0,\quad k\ge 2.\\
\end{aligned}
$$
Moreover,  $z\in\mc(L)$  if and only if $\phi^{(k)}([u,u]\otimes u \otimes {\dots} \otimes  u)=0$ for $k$ large enough.
\end{lemma}

\begin{remark}\label{universal} In particular, any element $z$ of degree $-1$ of a given $L_\infty$ algebra can be written as $\sum_{i\ge 1}\frac{1}{i!}\phi^{(i)}(u\otimes\dots\otimes u)$ and thus, independently of any finiteness or mildness assumption, Maurer-Cartan elements are not preserved in the standard fashion by $L_\infty$ morphisms. Note also that, even for $\phi$ mild, the condition $\phi^{(k)}([u,u]\otimes u \otimes {\dots} \otimes  u)=0$ for $k$ large enough is not automatically satisfied.
\end{remark}

\begin{corollary}\label{coromc} Let $L$ be an $L_\infty$ algebra of finite type. Then, an element $z\in L_{-1}$ is Maurer-Cartan if and only if there exists a mild $L_\infty$ morphism $\phi \colon (\lib (u),\partial)\to L$ such that $\phi^{(1)}(u)=z$ and $\phi^{(k)}(u\otimes\ldots\otimes u)=0$ for $k\ge 2$.
\end{corollary}
\begin{proof}
If $z\in \mc(L)$ the morphism $\phi$ of lemma above is obviously mild as  $\phi^{(k)}([u,u]\otimes u \otimes {\dots} \otimes  u)=0$ for $k$ large enough. Conversely, if $\phi$ is a mild $L_\infty$ morphism and $L$ is of finite type, then $\phi^{(k)}([u,u]\otimes u \otimes {\dots} \otimes  u)$ necessarily vanishes for $k$ large.
\end{proof}

\begin{proof}[Proof of Lemma \ref{existenciaMC}]
As $(\lib(u),\partial)$ is the vector space spanned only by $u$ and $[u,u]$, with $\partial u=-\frac{1}{2}[u,u]$, an
$L_\infty$ morphism $\phi \colon \lib (u)\to L$ is simply a CDGC morphism,
$$
\phi\colon(\Lambda(su,s[u,u]),\delta)\longrightarrow (\Lambda sL,\delta),
$$
which is completely determined by the elements
$$
 \phi^{(k)}(u\otimes {\dots} \otimes u),\quad \phi^{(k)}([u,u]\otimes u \otimes {\dots} \otimes  u),\quad k\geq 1,
 $$
satisfying the system $(1)$ of Section 1.
In this particular case, if
 we set
$$\phi^{(1)}(u)=z,\quad \phi^{(k)}(u\otimes {\dots} \otimes u)=0,\quad k\ge 2,
$$
and since $\ell_i=0$, for $i\ge 3$ in $(\lib(u),\partial)$,
 a direct computation shows that $\phi$ is indeed an $L_\infty$ morphism if the following identities hold for any $k\ge 1$,
$$\ell_k(z,\dots,z)=\binom{k}{2}\phi^{(k-1)}([u,u]\otimes
 u\otimes\cdots \otimes u)-\frac{k}{2}\phi^{(k)}([u,u]\otimes u\otimes\cdots
 \otimes
 u),$$
$$
\sum_{j=1}^k\binom{k-1}{j-1}\ell_j\Bigl(
\phi^{(k-j+1)}([u,u]\otimes u\otimes{\cdots}\otimes
u),z , \stackrel{j-1}{\dots},
z\Bigr)=0.\eqno{(4)}
$$

We will show that $\phi^{(k)}([u,u]\otimes u \otimes {\dots} \otimes  u)$, satisfying the above identities, are uniquely determined by the formula
$$
\phi^{(k)}([u,u]\otimes
u\otimes \cdots \otimes
u)=-2(k-1)!\sum_{i=1}^k\frac{1}{i!}\ell_i(z,
{\dots}, z).\eqno{(5)}
$$
First of all,  for $k=1$, the first  identity in $(4)$ is simply
$$
\ell_1z=-\frac{1}{2}\phi^{(1)}[u,u].
$$
Thus, we are forced to define
$$
\phi^{(1)}[u,u]=-2\ell_1(z),
$$
as in $(5)$. The second identity in $(4)$ for $k=1$ reads $\ell_1\phi^{(1)}[u,u]=0$ which is trivially satisfied:
$$
\ell_1\phi^{(1)}[u,u]=-2\ell_1^2 (z)=0.
$$
Assume the identities in $(4)$ are satisfied for $k-1$ by setting formula $(5)$  for integers smaller than $k$.

Again, from the first identity in $(4)$ for $k$, we are  forced to define
$$
\phi^{(k)}([u,u]\otimes u\otimes...
 \otimes
 u)=(k-1)\phi^{(k-1)}([u,u]\otimes
 u\otimes... \otimes u)-\frac{2}{k}
 \ell_k(z, ... , z).\eqno{(6)}
 $$
Now,  by the inductive hypothesis for $k-1$, this expression becomes
$$
\begin{aligned}
&-2(k-1)(k-2)!\sum_{i=1}^{k-1}\frac{1}{i!}\ell_i(z,\dots,z)-\frac{2}{k}
 \ell_k(z,\dots , z)\\
 &=-2(k-1)! \sum_{i=1}^{k}\frac{1}{i!}\ell_i(z,\dots, z),
 \end{aligned}
 $$
which is precisely the equation
  $(5)$ for $k$. To finish, we must check  that the second identity in $(4)$ for $k$,
  $$
\textstyle{\sum_{j=1}^k\binom{k-1}{j-1}\ell_j\Bigl(
\phi^{(k-j+1)}([u,u]\otimes u\otimes{\cdots}\otimes
u), z , \stackrel{j-1}{\dots},
z\Bigr)=0}
$$
holds.

For it,  replace in this equation $\phi^{(k-j+1)}([u,u]\otimes u\otimes{\cdots}\otimes
u)$ by its value on equation $(6)$ above for $k-j+1$. This yields the following, in which we have avoid the $\otimes$ sign for simplicity:
$$
\textstyle{\sum_{j=1}^k\binom{k-1}{j-1}\ell_j\Bigl(
\bigl((k-j)\phi^{(k-j)}([u,u] u {...}
u)-\frac{2}{k-j+1}\ell_{k-j+1}(z ,..., z)\bigr) ,z,
\stackrel{j-1}{...} ,z\Bigr).}$$
Then, this expression splits as
$$(k-1)\sum_{j=1}^{k-1}\binom{k-2}{j-1}\ell_j\Bigl(
\phi^{(k-j)}([u,u]u {\cdots}
u),z,\stackrel{j-1}{\dots},z\Bigr)$$
$$-\frac{2}{k}\sum_{j=1}^k\binom{k}{j-1}\ell_j(\ell_{k-j+1}(z,\dots,
z),z,\stackrel{j-1}{\dots}, z)).$$
By induction hypothesis the first summand is zero as it is the second identity in $(4)$ for $(k-1)$. The second summand is also zero by the $k^{\text{th}}$
higher Jacobi identity on $L$.

 Now we prove the second assertion. If $z\in \mc(L)$, then there is an integer $N$ such that $\ell_k(z,\dots , z)=0$ for $k\ge N$. Therefore, via equation $(5)$, and for $k\ge N$,
$$
\phi^{(k)}([u,u]\otimes
u\otimes \cdots \otimes u)=-2(k-1)!\sum_{i=1}^\infty \frac{1}{i!}\ell_i(z,
\stackrel{i}{\dots} , z)=0.
$$
The converse is also trivially satisfied in light of $(5)$.
\end{proof}

 In order to detect Maurer-Cartan elements at the cochain level, let $L$ be a mild $L_\infty$ algebra and let  $\{z_j\}_{j\in J}$  and $\{v_j\}_{j\in J}$ be basis of $L_{-1}$ and $V^0$ respectively (see Definition \ref{cocadena}). Then, any $z\in L_{-1}$, written as $z=\sum_j\lambda_jz_j$, is obviously identified with the  linear map $V^0\to\bk$ sending $v_j$ to $\lambda_j$ for all $j\in J$. However,  Maurer-Cartan elements of $L$ are not, in general, those $z$ for which this map can be extended as an augmentation of the cochains, i.e., as a CDGA morphism $\cinfinito(L)\to\bk$. The following examples corroborates this assertion.

 \medskip

\begin{example}\label{aumenta} (1) Let $L$ be the mild $L_\infty$ algebra generated by $\mathcal B=\{ \omega_i , \nu \}_{i\ge 2} $, with $|\omega_i|=-2$, $|\nu |=-1$, and where the only non zero brackets on generators are:
$$
\ell_1(\nu )=-\omega_2,\qquad \ell_k(\nu,\dots,\nu)=k!(\omega_k-\omega_{k+1}),\quad k\ge 2.
$$
Then, $\cinfinito(L)=(\Lambda V,d)$ in which $V$ is generated by $\{v,u_i\}_{i\ge 2}$, with $|v|=0$, $|u_i|=-1$,  $dv=0$ and $du_i=v^i-v^{i-1}$ for $i\ge 2$. Observe that the morphism $\cinfinito(L)\to\bk$ sending $u_i$ to $0$ for all $i$ and $v$ to $1$ is a well defined augmentation, even though $\nu$ is not a Maurer-Cartan element.

\medskip

(2) On the other hand, consider $L=L_{-1}$ an abelian $L_\infty$ algebra, of infinite dimension and concentrated in degree $-1$. Thus $\mc(L)=L$. Observe that  $\cinfinito(L)=(\Lambda V,0)$ where $V=V^0$ is of the same dimension as $L$. Thus, the set $\aug\cinfinito(L)$ of augmentations $\cinfinito(L)\to\bk$ is in one-to-one correspondence with the set of linear maps $V\to\bk$, i.e., with $V^\sharp$ which has dimension strictly bigger than that of $L$. Thus the cardinality of $\aug\cinfinito(L)$ is bigger than the cardinality of $L$. For instance, the augmentation $\cinfinito(L)\to\bk$ which sends any element of a given basis of $V$ to $1$ does not correspond to any  Maurer-Cartan element of $L$.
\end{example}

\medskip

\begin{remark}\label{tipofinito} In light of  example (2) above, it is important to note that, if one considers non-finite type mild $L_\infty$ algebras, very special and technical restrictions are needed to identify the Maurer-Cartan set with the   augmentations of the cochain algebra. For clarity in the exposition, these conditions, and the particular class of augmentations which correspond to Maurer-Cartan elements, are made explicit in the final remark at the end of the section. In the same way, in view of Lemma \ref{existenciaMC} and Remark \ref{universal},  Maurer-Cartan elements are not preserved by mild $L_\infty$ morphisms unless either finite type is assumed, or again, special restrictions are applied. Thus, hereafter, and again for the sake of clearness,  we restrict $\catlinfinito$ to the class of mild, finite type $L_\infty$ algebras.  Nevertheless, the reader may keep in mind that, in the general case, all of what follows remains true only under the assumptions in Remark \ref{finalre}.
\end{remark}

\begin{definition}\label{mc(g)} Let $g\colon L\to L'$ be a morphism in $\catlinfinito$ and   $z\in\mc(L)$. Define the map
$\mc(g)\colon \mc(L)\longrightarrow \mc(L')$ by
$$
 \mc(g)(z)=\sum_{k\ge1}\frac{1}{k!}g^{(k)}(z\otimes{\cdots}\otimes z).$$
 \end{definition}

 In the next result we see that $\mc(g)$ is well defined.
 Moreover, with the finiteness type assumptions in the above  remark, we  identify the Maurer-Cartan elements of $L\in\catlinfinito$ in a functorial way with the set $\aug\cinfinito(L)$ of augmentations of $\cinfinito(L)$.  We stress here that, to our knowledge,  the following result and the Corollary \ref{funtormc} that follows are not straightforward and do not follow at once  by simply generalizing their classical DGL counterpart of \cite[Remark 16]{bufemu} or \cite[Proposition 1.1]{getz} (compare to \cite[Lemma 2.3]{ber} or \cite[Proposition 2.2]{la}).

\begin{proposition}\label{inicial3} Let $g\colon L\to L'$ be a morphism in $\catlinfinito$ and   $z\in\mc(L)$. Then, $\mc(g)(z)$ is indeed a Maurer-Cartan element in $L'$. That is,
 $$\sum_{k\ge1}\frac{1}{k!}g^{(k)}(z\otimes{\cdots}\otimes z)\in\mc(L').$$
 Moreover, the functor
$$
\mc\colon\catlinfinito\to\catset
$$
 is naturally equivalent to the functor
 $$
 \aug\colon \catlinfinito\to\catset
 $$
 which assigns to  $g\colon L\to L'$  the map $\aug(g)\colon \aug \cinfinito(L)\to\aug\cinfinito (L')$ given by composition, $\aug(g)(\varepsilon)=\varepsilon\cinfinito(g)$.

 \end{proposition}

 We prove it from a geometrical perspective and from a {\em based} point of view:

\begin{definition}    A {\em based augmentation} of a given CDGA  $A$ is a morphism $A\to (\Lambda (x,y),d)$ where, as above,  $(\Lambda (x,y),d)=\cinfinito(\lib(u),\partial)$, that is,
 $x$ and
$y$ are generators of degrees $0$ and $-1$ respectively,  $dx=0$, and $dy=\frac{1}{2}(x^2-x)$.
 \end{definition}

 Observe that,  if we compose any based augmentation of $A$ with\break $\rho \colon (\Lambda (x,y),d)\to\bk$, where  $\rho(x)=1$, we obtain a classical augmentation $A\to \bk$.  Conversely, we have the following.

\begin{lemma}\label{existenciaaum} Let $(\Lambda V,d)$ be a free CDGA and let
 $\Phi\in\Lambda ^+x$ such that $\rho(\Phi)=1$. Then, any augmentation  $f\colon (\Lambda V, d)\to\bk$ has a unique lifting $ f_{\Phi}$ to $(\Lambda (x,y),d)$ such that, for any $v\in V^0$,
$$
 f_\Phi(v)={f(v)}\Phi.
$$
\end{lemma}
\begin{proof}
For degree reasons we set $ f_\Phi$ to be zero in $V^{\ge1}$ and $V^{\le-2}$.   Let $w\in V^{-1}$ and write
$dw=\alpha+\beta$, where $\alpha\in \Lambda^+V^0$ and $\beta\in \Lambda^+V^{\not=0}\cdot
(\Lambda V)$. Then,  $f(dw)= f(\alpha)$. Write
$\alpha=p(v_1,\dots ,v_n)$  as  a polynomial without constant term in the
generators of $V^0$, and set $\lambda_i=f(v_i)$, for $i=1,\ldots,n$. Then,
$$
p(\lambda_1,\dots ,\lambda_n)=fd(w)=df(w)=0.
$$
On the other hand,
$$
 f_\Phi(dw)=p\big({\lambda_1}\Phi,\dots
,{\lambda_n}\Phi\bigr)=P(x),
$$
which is a  polynomial in $x$ without constant term, and it satisfies
$P(1)=p(\lambda_1,\dots ,\lambda_n)$ $=0$. Hence $P(x)=x(x-1)r(x)$ and we define
$ f_\Phi(w)=2yr(x)$ so that $d f_\Phi(w)= f_\Phi(dw)$.

Finally, we check that, for any generator $u\in V^{-2}$, $ f_\Phi(du)=0$. Indeed, write  $ f_\Phi(du)=yQ(x)$ whose differential
$\frac{1}{2}(x^2-x)Q(x)$ has to vanish. Thus $Q(x)=0$ and the lemma holds.
\end{proof}

\begin{remark}\label{aumentacionbasada}
 The definition of based augmentation comes from its geometric counterpart of choosing a non basepoint $x_1$ of a space $(X,x_0)$ in the based category.  Moreover, the lemma above exhibit the commutative diagram
$$
\xymatrix{&\Lambda(x,y)\ar[d]^{\rho}&&&(S^0,1)\ar[ld]\\
\Lambda V\ar[r]_f\ar[ru]^{{f}_\Phi}&\mathbb{Q}&\text{as a model of}&(X,x_0)&x_1.\ar[l]\ar[u]}
$$
 \end{remark}

\begin{proof}[Proof of Proposition \ref{inicial3}] We first show that there is a natural bijection
$$\mc(L)\cong\aug\cinfinito(L).$$
Choose a basis $\{z_j\}_{j=1}^m$ of $L_{-1}$,  set $\cinfinito(L)=(\Lambda V,d)$ with $V=(sL)^\sharp$, and for each $j$ denote by $v_j$ the element $(sz_j)^\sharp$ of $V^0$.

Given $z\in \mc(L)$, write $z=\sum_{j=1}^m\lambda_jz_{j}$ and apply Corollary \ref{coromc} (recall that $L$ is assumed to be of finite type) to obtain the mild $L_\infty$ morphism $\phi\colon \lib(u)\to L$  for which $\phi^{(1)}(u)=z$, $\phi^{(n)}(u\otimes\dots\otimes u)=0$ for $n\ge 2$. Then, since $\phi$ is mild, we can construct the based augmentation $\cinfinito(\phi)\colon (\Lambda V,d)\to (\Lambda (x,y),d)$  which sends each $v_{j}$ to $\lambda_jx$. Therefore, the composition\break $\rho\cinfinito(\phi)\colon \cinfinito(L)\to\bk$ is an augmentation denoted by $\varepsilon_z$.

Conversely, consider any  augmentation $\varepsilon\colon (\Lambda V,d)\to \bk$  and set $\varepsilon(v_{j})=\lambda_j$.  Lift $\varepsilon$ via Lemma  \ref{existenciaaum} to a based augmentation $\varepsilon_x\colon (\Lambda V,d)\to (\Lambda (x,y),d)$. Then, observe that $\varepsilon_x=\cinfinito(\phi)$ for a mild  $L_\infty$ morphism $\phi\colon\lib(u)\to L$ in which  $\phi^{(1)}(u)=\sum_{j=1}^m\lambda_jz_{j}$ and $\phi^{(n)}(u\otimes\dots\otimes u)=0$ for $n\ge 2$. Since $L$ is of finite type,  again by Corollary \ref{coromc}, the element $z=\sum_{j=1}^m\lambda_jz_{j}$ is a Maurer-Cartan element of $L$.

 Thus, the correspondence
$z\leftrightarrow \varepsilon_z$
establishes the asserted bijection.

Next, we prove the first assertion of the proposition by showing that, given $g\colon L\to L'$ a morphism in $\catlinfinito$, then
$$
\mc(g)\colon \mc(L)\to\mc(L')$$ is identified with
$$\aug(g)\colon \aug \cinfinito(L)\to\aug\cinfinito (L').$$

For it, let $z\in MC(L)$. By the bijection $\mc(L)\cong\aug\cinfinito(L)$, the  Maurer-Cartan element $z$ corresponds to the augmentation in $\aug\cinfinito(L)$ given by
$$
\rho\cinfinito (\phi)
$$
where
 $\phi\colon\lib(u)\to L$ is the mild $L_\infty$ morphism, obtained via Corollary \ref{coromc},  corresponding to $z\in \mc(L)$. Applying $\aug(g)$ to this augmentation we obtain,
 $$
 \aug(g)\bigl(\rho\cinfinito (\phi)\bigr)=\rho\cinfinito (\phi)\cinfinito(g)=\rho\cinfinito (g\phi )\in\aug(L').
 $$
 We will prove that this augmentation corresponds, via again the bijection $\mc(L')\cong\aug\cinfinito(L')$, with the element $\frac{1}{k!}\sum_k g^{(k)}(z\otimes\dots\otimes z)\in L'_{-1}$ which must be then a Maurer-Cartan element in $L'$ as stated.

 For it, we need to lift this augmentation $\rho\cinfinito (\phi g)$, via Lemma \ref{existenciaaum}, to a based augmentation $\varepsilon_x\colon\cinfinito(L')\to(\Lambda(x,y),d)$. Observe that $\varepsilon_x$ is, in general, far from being $\cinfinito(g\phi)=\cinfinito(\phi)\cinfinito(g)$. Indeed, although the image of $\cinfinito(\phi)$ on degree zero elements is linear on $x$, the image of $\cinfinito(g)$ may not be linear on degree zero elements. Let us then describe explicitly $\varepsilon_x$.

    Choose finite basis $\{z_j\}_{j\in J}$, $\{z_i'\}_{i\in I}$ of $L_{-1}$ and $L'_{-1}$ respectively and write $\cinfinito(L)=\Lambda W$, $\cinfinito(L')=\Lambda V$. Observe that  $W^0$ and $V^0$ are generated by $\{w_j\}_{j\in J}$  and $\{v_i\}_{i\in I}$ where $w_j=(sz_j)^\sharp$ and $v_i=(sz'_i)^\sharp$ for each $i\in I$ and $j\in J$.

    If $z=\sum_j\lambda_jz_j$, then $\cinfinito(\phi)\colon\cinfinito(L)\to(\Lambda(x,y),d)$ is defined on $W^0$ by $\cinfinito(\phi)w_j=\lambda_jx$.

    On the other hand, write $\cinfinito(g)=\sum_{k\ge 1}\cinfinito(g)_k$ with $\cinfinito(g)_kV\subset\Lambda^kW$ and set
    $$
    \cinfinito(g)_k(v_i)=P_{ik}+Q_{ik},\,\,\text{with}\,\,\, P_{ik}\in\Lambda^kW^0\,\,\text{and}\,\,\,Q_{ik}\in \Lambda^+W^{\not=0}\cdot\Lambda W.
    $$
    Then,
    $$
   \rho\cinfinito(\phi)\cinfinito(g) _k(v_i)=\rho\cinfinito(\phi)(P_{ik})=P_{ik}(\lambda_j),
    $$
    where $P_{ik}(\lambda_j)$ is the scalar obtained by evaluating the ``polynomial" $P_{ik}$ on the $\lambda_j$'s. Thus, $\varepsilon_x$ is defined on $V^0$ as,
    $$
    \varepsilon_x(v_i)=\sum_{k\ge 1}P_{ik}(\lambda_j)x,
    $$
    being this a finite sum due to the mildness assumption.

    Now that we have explicitly precised the lifting $\varepsilon_x$ of the augmentation $\rho\cinfinito (g\phi )$, we need to identify the Maurer-Cartan element $z'$ that it represents. By the first part of the present proof, this element is precisely,
    $$
    z'=\sum_i\bigl(\sum_k P_{ik}(\lambda_j)\bigr)z'_i.
    $$

On the other hand,
an easy computation shows  that
    $$
    \langle\cinfinito(g)_kv_i;sz,\ldots,sz\rangle=k!P_{ik}(\lambda_j)
    $$
which, in light of
 $(3)$ of Section 1, let us conclude that
    $$
    P_{ik}(\lambda_j)=\frac{1}{k!}\langle v_i;sg^{(k)}(z\otimes\dots\otimes z)\rangle.
    $$
     Therefore,
    $$z'=\sum_{i,k}P_{ik}(\lambda_j)z'_i=\sum_{i,k}\frac{1}{k!} \langle v_i;sg^{(k)}(z\otimes\dots\otimes z)\rangle z'_i=
    \sum_k \frac{1}{k!}g^{(k)}(z\otimes\dots\otimes z)
    $$
     and the proposition is proved.
\end{proof}

  \begin{corollary}\label{funtormc} Let $L\in\catlinfinito$ and $A\in\catcdga$ such that $L\otimes A$ is of finite type. Then,   there is a bijection
$$
\mc(L\otimes A)=\catcdga(\cinfinito(L),A).
$$
\end{corollary}

\begin{proof} As $L\otimes A$ is mild and of finite type, apply
Proposition \ref{inicial3}, to identify  a given Maurer-Cartan element $z$ of $L\otimes A$  with an augmentation
$$
\varepsilon_z\colon \cinfinito(L\otimes A)\cong\Lambda(sL\otimes A)^\sharp\to\bk.
$$
This produces a  degree zero  linear map $(sL)^\sharp\to A$ which is  extended to an algebra morphism $\cinfinito (L)\to A$. A straightforward computation shows that  it commutes with differential since $\varepsilon_z$ does. Conversely, any  CDGA morphism $\cinfinito(L)\to A$ gives rise, by the procedure above, to an  augmentation $\cinfinito(L\otimes A)\to\bk$.
 \end{proof}

 It is important also to observe that if $L\otimes A$ fails to be of finite type, and even if $L$ and $A$ are, $\mc(L\otimes A)$ is no longer identified with the set of morphisms $\catcdga(\cinfinito(L),A)$ as shown in the following example. In the general case, as in Remark \ref{tipofinito}, it is necessary to impose technical finiteness restrictions in  the class of morphisms. Again, this is explicitly detailed in the Remark \ref{finalre} below.

 \begin{example}\label{morfi}
 Let $L=\sum_{n<0}L_{2n+1}$ be an abelian $L_\infty$ algebra (i.e., all brackets are zero) concentrated in odd negative degrees, with  $L_n$ of dimension $1$ for all $n$, and let $A=(\Lambda x,0)$ be the polynomial algebra on a single generator of degree $2$, without constant terms. Clearly $\mc(L\otimes A)=(L\otimes A)_{-1}$ which is of infinite countable dimension. On the other hand, $\cinfinito(L)=(\Lambda(y_0,y_2,y_4,\ldots),0)$ and thus, $\catcdga(\cinfinito(L),A)$ is of infinite, uncountable dimension.
\end{example}

\begin{remark}\label{finalre} In non finite type mild $L_\infty$ algebras, Maurer-Cartan sets are identified with a very special class of augmentations which we now describe. For it, let $L$ be a mild $L_\infty$ and let $f\colon(\Lambda V,d)\to\bk$ be an augmentation of $\cinfinito(L)=(\Lambda V,d)$. Choose a complement $W\subset V^0$ of $\ker f_{|_{V_0}}$ and write   $\Lambda V=(\Lambda W)\oplus B$. We say that $f$ is a {\em Maurer-Cartan augmentation} if it has finite dimensional support, i.e.,  if $W$ is finite dimensional, and there is an integer $k\ge 0$ such that $dV^{-1}\subset (\Lambda^{\le k}W)\oplus B$.

Then the same proof as in Proposition \ref{inicial3} shows that $\mc(L)$ are in bijective correspondence with Maurer-Cartan augmentations of $\cinfinito(L)$.

More generally, let $\varphi\colon (\Lambda V,d)\to A$ be a CDGA morphism and let  $W\subset V$ a complement of $\ker\varphi_{|_{V}}$. We say that $\varphi$ is a {\em Maurer-Cartan morphism} if it has finite dimensional support, i.e., if $W$ is finite dimensional, and there is an integer $k\ge 0$ such that $dV\subset (\Lambda^{\le k}W)\oplus B$. Again, the proof in Corollary \ref{funtormc} shows that $\mc(L\otimes A)$ are in one-to-one correspondence with Maurer-Cartan morphisms from $\cinfinito(L)$ to $A$.

Finally, as previously remarked, Lemma \ref{existenciaMC} shows that $\mc$ does not define a functor unless additional restrictions are applied. If finite type is not assumed, at least one needs to consider mild $L_\infty$ algebras for which Im$\,\ell_j$ is of finite type for all $j$, and mild $L_\infty$ morphisms for which, additionally, Im$\,\phi^{(j)}$ are also of finite type for all $j$. With these restrictions, the functoriality asserted in Proposition \ref{inicial3} remains valid.

\end{remark}

\medskip

\section{The Lawrence-Sullivan construction as a \break transferred $\infty$-structure}

The Lawrence-Sullivan construction introduced in \cite{lasu}  will play a fundamental role in the understanding of the notion of homotopy in $\catlinfinito$. Thus, it deserves to be carefully presented.

Given $V$ a graded vector space   $\lib(V)$ denotes the {\em free lie algebra} generated by $V$. If, in the tensor algebra $T(V)=\sum_{n\ge0}T^n(V)$, we consider the Lie structure given by commutators, $\lib(V)$  is the Lie subalgebra generated by $V$. Replacing $T(V)$ by the {\em complete tensor algebra} $\widehat{T}(V)=\Pi_{n\ge0}T^n(V)$, we obtain $\widehat\lib(V)$, the {\em complete free Lie algebra} generated by $V$. A generic element of $\widehat T(V)$ will be written as a formal series $\sum_{n\ge 0}\phi_n$ with $\phi_n\in T^n(V)$. Note that $T(V)\subset \widehat T(V)$ and $\lib(V)\subset\widehat\lib(V)$. The universal enveloping algebra $U\lib(V)$ of  $\lib(V)$ extends to the complete  free Lie algebra to produce a graded algebra $\widehat U\widehat \lib(V)=\widehat T\bigl(\widehat\lib(V)\bigr)/\sim$ naturally isomorphic to $\widehat T(V)$.

\begin{definition} \label{lsconsstruction} The {\em Lawrence-Sullivan construction}, denoted by $\lasu$, is the complete free DGL $(\widehat\lib(a,b,x),\partial)$ in which $a$ and $b$ are Maurer-Cartan  elements and
$$
\partial(x)=[x,b]+\sum_{i=0}^\infty\frac{B_i}{i!}\ad_x^i(b-a),
 $$
 where $B_i$ denotes the $i^{\text{th}}$ Bernoulli number. Equivalently, as shown in \cite{lasu},
 $$\partial x=\ad_x(b)+h_x(b-a),$$
 where,  as operator,
 $$
 h_x=\frac{\ad_x}{e^{\ad_x}-{\rm id}}.
 $$
 Another inductive description of the differential in this complete free Lie algebra was suggested in \cite{lasu} and shown to be equivalent to the above in \cite[Main Theorem]{patan}.
 \end{definition}
Recall that
 Bernoulli numbers can be recursively defined by $B_0=1$  and
$$-\frac{B_n}{n!}=\sum_{i=0}^{n-1}\frac{B_{n-1-i}}{(n-1-i)!(i+2)!}\,\,,\qquad n\ge 1.\eqno(7)$$
 As a differential graded algebra, DGA henceforth, the universal enveloping algebra of $\lasu$ is also known.

 \begin{theorem}{\em \cite[Theorem 3.3]{bumu}}\label{cilindro}  $U\lasu$ is the ``complete cylinder"   $\cyl T(a)=(\widehat T(a\oplus b\oplus x),d)$ in which $|a|=|b|=-1$, $da=-a\otimes a$, $db=-b\otimes b$ and $x$ is a degree zero element with
 $$
dx= x\otimes b-b\otimes x+ \sum_{n\ge 0} \sum_{p+q=n}(-1)^{q}\frac{B_{n}}{p!q!}x^{\otimes p}\otimes(b-a)\otimes x^{\otimes q}\eqno{\square}
$$

\end{theorem}

We show  in this section how these objects arise naturally as transferred infinity structures from the dual of the standard acyclic differential graded algebra. For it, special properties satisfied by  cocommutative $A_\infty$ coalgebras are needed.

 Recall that an \emph{$A_{\infty}$ coalgebra} is a graded vector space $C$
together with a family $\{\Delta_k\}_{k\ge 1}$ of  degree $k-2$  linear
maps, $\Delta_k\colon C\to C^{\otimes k}$, such that
$$
\sum_{k=1}^i\, \sum_{n=0}^{i-k}(-1)^{k+n+kn}({\rm id}_C^{\otimes
i-k-n}\otimes \Delta_k \otimes {\rm id}_C^{\otimes n})\Delta_{i-k+1}=0.
$$

An $A_\infty$ coalgebra is {\em cocommutative} if $\tau\circ \Delta_k=0$ for  every $k\ge 1$. Here $\tau\colon T(C)\to T(C)\otimes T(C)$ is the {\em unshuffle coproduct}, that is,
$$
\tau(a_1\otimes\dots\otimes a_n)=\sum_{i=1}^n\sum_{\sigma\in S(i,n-i)}\varepsilon_\sigma (a_{\sigma(1)}\otimes\dots\otimes a_{\sigma(i)})\otimes(a_{\sigma(i+1)}\otimes\dots\otimes a_{\sigma(n)}).
$$

Observe that $A_\infty$ coalgebras are in one-to-one correspondence with differentials in the augmentation kernel $\widehat T^+(s^{-1}C)=\Pi_{n\ge 1}T^n(s^{-1}C)$ of the complete tensor algebra  $\widehat T(s^{-1}C)$ on the desuspension of $C$, $(s^{-1}C)_p=C_{p+1}$. Indeed,  such a differential $d$   is determined by its image  on $s^{-1}C$, which is written as a sum $d=\sum_{k\ge 1}d_k$, with $d_k(s^{-1}C)\subset T^{ k}(s^{-1}C)$, for $k\ge 1$. Then, the operators $\{\Delta_k\}_{k\ge 1}$ and $\{d_k\}_{k\ge 1}$ uniquely determine each other via
$$
\begin{aligned}
\Delta_k&=-s^{\otimes k}\circ d_k\circ s^{-1}\colon C\to
C^{\otimes k},\\
d_k&={-}(-1)^{\frac{k(k-1)}{2}}(s^{-1})^{\otimes k}\circ\Delta_k\circ s\colon s^{-1}C\to T^{ k}(s^{-1}C).
\end{aligned}\eqno{(8)}
$$
Note that $d$ is uniquely extended  to $\widehat T(s^{-1}C)$ by setting $d1=0$.

In the same way, cocommutative $A_\infty$ coalgebras are in one-to-one correspondence with differentials on $\widehat T(s^{-1}C)$ for which the image of each $d_k$ lies in the invariants of $T^{k}(s^{-1}C)$ under the graded action of the symmetric group. In what follows we denote by $\cobarinf(C)$ the differential graded algebra $(\widehat T(s^{-1}C),d)$ corresponding to the given $A_\infty$ structure on $C$.
A {\em morphism} of $A_{\infty}$ coalgebras or {\em $A_\infty$-morphism}
between  $C$ and $C'$ is a differential graded algebra morphism
$$
f\colon \cobarinf(C)=(\widehat T(s^{-1}C),d)\to (\widehat T(s^{-1}C'),d')=\cobarinf(C').
$$
 This is equivalent to the
existence of a family of  maps $f_{(k)}\colon C\to
C'^{\otimes k}$, $k\ge 1$, of degree $k-1$, which satisfy the equations  involving the operators
$\{\Delta_k\}$ and $\{\Delta_k'\}$ and arising from the equality $d'f=fd$. We often denote an  $A_\infty$ morphism  simply by $f\colon C\to C'$.  An $A_{\infty}$-morphism is a
\emph{quasi-isomorphism} if $f_{(1)}\colon (C,\Delta_1)\stackrel{\simeq}{\to}(C',\Delta'_1)$ is a quasi-isomorphism of differential graded vector spaces.

Naturality of the constructions above exhibits
$$
\cobarinf\colon\catcono\to \catdga
$$ as a functor from the category of  $A_\infty$ coalgebras to the category of differential graded algebras.

Note that a (cocommutative)   differential graded coalgebra $C$ is simply a (cocommutative) $A_\infty$ coalgebra such that $\Delta_k=0$ for $k\ge 3$. In this case $\cobarinf(C)$ is the classical cobar construction on $C$.

On the other hand,  recall that the  {\em Quillen construction} on a given CDGC $(C,\delta)$ is the DGL defined by $(\lib(s^{-1} C),\partial)$ in which $\partial=\partial_1+\partial_2$, where $\partial_1(s^{-1}c)=-s^{-1}\delta c$ and $$
\partial_2(s^{-1}c)=\frac{1}{2}\sum_i (-1)^{|a_i|}[ s^{-1}a_i,s^{-1}b_i],
$$
with $c\in {C}$ and ${\Delta }c=\sum_i
a_i\otimes b_i$. This can be easily extended to any cocommutative $A_\infty$ coalgebra:

 Indeed,
any cocommutative $A_\infty$ coalgebra $C$ induces a natural DGL structure on $\widehat\lib (s^{-1}C)$. The differential $\partial=\sum_{k\ge 1}\partial_k$ with $\partial_k\colon s^{-1}C\to \lib^k (s^{-1}C)$,
is determined by $\Delta_k$ in the same way as the classical  Quillen construction.

This defines the {\em generalized Quillen functor},
$$
\mathcal{L}\colon \catco \to \catdgl,\quad \mathcal{L}(C)=(\widehat\lib (s^{-1}C),\partial)
$$
which preserves quasi-isomorphisms and  whose composition with the completed universal enveloping algebra functor is precisely the $\cobarinf$ construction functor,
$$
\widehat U\mathcal{L}=\cobarinf.
$$

Next, we will consider a particular cocommutative differential graded coalgebra, CDGC henceforth, which will play a ``universal"  role analogous to the one of  the standard acyclic commutative differential graded algebra   $A=\Lambda (t,dt)$, free as a commutative graded algebra,  with $|t|=0$. The dual $A^\sharp$ of this CDGA is a  differential vector space also concentrated in (subscript) degrees $0$ and $1$. With the convention of Section 1, we write any vector of $A^\sharp_0$ as a formal series $\sum_{j\ge 0}\lambda_j\alpha_j$ representing the function $\{
t^j\}_{j\ge 0}\to\bk$  which assigns  to each $t^j$  the scalar $\lambda_j$. Respectively, any element in  $A^\sharp_1$ will be written as $\sum_{j\ge 0}\mu_j\beta_j$ representing the map $\{t^jdt\}_{j\ge 0}\to \bk$ sending $t^jdt$ to $\mu_j$. Observe that, with this notation, the differential $\delta$ in $A^\sharp$ is the only linear and formal extension to the above series for which $\delta(\beta_j)=(j+1)\alpha_{j+1}$ and $\delta(\alpha_j)=0$, for all $j\ge 0$. However, Since $A$ is not of finite type, $A^\sharp$ does not inherits  a coalgebra structure. Indeed, if we set, for each $j\ge 0$,
$$\Delta (\alpha_j)=\sum_{k=0}^{j}\alpha_k\otimes
\alpha_{j-k},\quad \Delta
(\beta_j)=\sum_{k=0}^{j}\beta_k\otimes
\alpha_{j-k}+\sum_{k=0}^{j}\alpha_k\otimes \beta_{j-k},$$
and extend $\Delta$ linearly to any series in $A^\sharp$, the resulting map lands in the completed tensor product, $\Delta\colon A^\sharp\to A^\sharp\widehat\otimes A^\sharp$.

Thus, we proceed as follows: consider the decreasing sequence $B_1\supset B_2\supset\dots\supset B_n\supset\dots$ of subspaces of $A^\sharp$,
$$
B^1=\Delta^{-1}(A^\sharp\otimes A^\sharp),\qquad B^n=\Delta^{-1}(B^{n-1}\otimes B^{n-1}),
$$
\begin{definition}\label{universal2} Define the {\em universal  cocommutative differential graded coalgebra} as $\basum=\cap_{n\ge 1} B_n$. This CDGC contains the vector space generated by $\{\alpha_j,\beta_j\}_{j\ge 0}$. Observe that $\alpha=\sum_{j\ge 0}\alpha_j$ is a counit of $\basum$ as $\Delta\alpha=\alpha \otimes \alpha$. Note also that $\beta=\sum_{j\ge 0}\beta_j\in \basum$ as $\Delta \beta=\alpha\otimes\beta+\beta\otimes\alpha$.
 \end{definition}

\begin{remark}\label{dobledual} Observe that the  map  $\Lambda(t,dt)\hookrightarrow {\mathfrak B}^\sharp$ sending $t^j$ and $t^jdt$ to $\alpha_j^\sharp$ and $\beta_j^\sharp$ respectively, for each $j\ge 0$  is a CDGA morphism.
\end{remark}

On the other hand, consider the  cocommutative $A_\infty$ coalgebra $C$ whose $\infty$-cobar construction  $\cobarinf(C)$ is $\cyl T(u)$. In other words,
$C=\langle y,z,c\rangle$ in which $y={sb}$ and $z={sa}$ are of degree $0$ and $c={-sx}$ is of degree $1$. The operators $\{\Delta_k\}_{k\ge1}$ are immediately deduced using equation $(8)$ and the differential on $\cyl T(u)$ given in Theorem \ref{cilindro}:
 $$
 \begin{aligned}
&\Delta_1c=y-z,\quad \Delta_1y=\Delta_1z=0,\\
&\Delta_2c= -\frac{1}{2}c\otimes (y+z)-\frac{1}{2}(y+z)\otimes c,\quad\Delta_2y=-y\otimes y,\quad \Delta_2z=-z\otimes z,\\
&\Delta_kc=\sum_{p+q=k-1}\frac{B_{k-1}}{p!q!}c^{\otimes p}\otimes(y-z)\otimes c^{\otimes q},\quad \Delta_ky=\Delta_kz=0,\quad k\ge 3.\quad{(9)}
\end{aligned}
$$

Next, define maps of differential vector spaces,
$$
\xymatrix{({\mathfrak B},\delta) \ar@<0.75ex>[r]^-\theta& (C,\Delta_1) \ar@<0.75ex>[l]^-\omega },
$$
as follows:
set $\theta(\alpha_0)=-y$, $\theta(\alpha_1)=y-z$, $\theta(\alpha_j)=0$ for $j\ge 2$,
$\theta(\beta_0 )=c$ and $\theta(\beta_j)=0$ for $j\ge 1$. Then, extend $\theta$ linearly to any series in ${\mathfrak B}$. On the other hand, define $\omega(y)=-\alpha_0$, $\omega(z)=-\sum_{i\ge 0} \alpha_i$ and $\omega(c)=\sum_{i\ge 0}\frac{1}{i+1}\beta_i$. As both ${\mathfrak B}$ and $C$ are acyclic,  a simple inspection shows that $\omega$ and $\theta$ are quasi-isomorphisms and $\theta\omega={\rm id}_C$.

 The main result in this section reads as follows.

 \begin{theorem}\label{principal2} There are quasi-isomorphisms of $A_\infty$ coalgebras,
 $$
\xymatrix{{\mathfrak B} \ar@<0.75ex>[r]^-\Theta& C\ar@<0.75ex>[l]^-\Omega },
$$
such that $\Theta_{(1)}=\theta$ and $\Omega_{(1)}=\omega$. In other words, $\cyl T(u)$ is a quasi-isomorphic retract of the classical cobar construction on ${\mathfrak B}$.
\end{theorem}

Applying the functor $\mathcal{L}$, we obtain the following.

\begin{corollary}\label{lie}   There are quasi-isomorphisms of DGL's
$$
\xymatrix{\mathcal{L}({\mathfrak B}) \ar@<0.75ex>[r]^-{\mathcal{L}(\Theta)}& \lasu\ar@<0.75ex>[l]^-{\mathcal{L}(\Omega)} }.\eqno{\square}
$$
\end{corollary}

The proof of Theorem \ref{principal2}, as expected, consists in a careful application of   the classical {\em Perturbation Lemma}  (see for instance \cite{huebschka} or \cite[\S3]{gulambs}), known nowadays by the {\em Homotopy Transfer Theorem} \cite{kontsoi,lova}, in the context of $A_\infty$ structures, and that we now recall.

\begin{theorem}\label{perturbacion}
Let
$$
\xymatrix{ \ar@(ul,dl)@<-5.5ex>[]_K  & (M,d) \ar@<0.75ex>[r]^-\theta & (N,d) \ar@<0.75ex>[l]^-\omega }
$$
be a diagram of complexes in which $(M,d)$ is a DGC with diagonal $\Delta$. Assume that $\theta\omega={\rm id}_N$, $\theta K=K\omega=K^2=0$ and $K$ is a chain homotopy between ${\rm id}_M$ and $\omega\theta$, i.e., $Kd+dK=\omega\theta-{\rm id}_M$. Then, there is an explicit $A_\infty$ coalgebra structure on $N$ and morphisms of $A_\infty$ algebras,
$$
\xymatrix{M \ar@<0.75ex>[r]^-\Theta& N\ar@<0.75ex>[l]^-\Omega },
$$
such that $\Theta_{(1)}=\theta$ and $\Omega_{(1)}=\omega$. \end{theorem}

We also briefly recall, for each $k\ge 2$,  the explicit $k$-diagonal $\Delta_k\colon N\to N^{\otimes k}$ obtained in theorem above.
Let $PT_k$ be the set of isomorphism classes of
planar rooted trees with internal vertices of valence two and
exactly $k$ leaves. For each  tree $T\in PT_k$, we define
a linear map $\Delta_T\colon N\to N^{\otimes k}$ as follows: label the leaves, internal edges, internal vertices and the root of the tree $T$ by $\theta$, $K$, $\Delta$ and $\omega$ respectively. Then, $\Delta_T$ is defined as the composition of the different labels moving
up from the root to the leaves, or terminal edges, of $T$.
Then, $\Delta_k$ is defined as
$$
\Delta_k=\sum_{T\in PT_k}\varepsilon_T\Delta_T
$$
where
$\varepsilon_T$ is  determined by the parity of the
number of pairs $(\ell ,v)$ where $\ell $ is a left terminal edge in the vertex $v$ of $T$, and $v$ has an even
number of incoming edges \cite[\S3]{ber2}.

\begin{proof}[Proof of Theorem \ref{principal2}]
Define a degree $-1$ map $K\colon {\mathfrak B}\to {\mathfrak B}$ as follows: set $K(\alpha_0)=0$, $K(\alpha_1)=\sum_{i\ge 1}\frac{1}{i+1}\beta_i$ and $K(\alpha_j)={-}\frac{1}{j}\beta_{j-1}$ for $j\ge 2$. Then, extend $K$ linearly to any series in ${\mathfrak B}_0$ and define it to be zero on ${\mathfrak B}_1$.  A simple inspection shows that the maps $K,\theta,\omega$ satisfy the assumptions on the theorem above so it remains to show that, for each $k\ge 2$, the transferred $k^{\text{th}}$ diagonals on $C$ coincide with the ones in equation $(9)$.

First, for  $k=2$, this is a short computation:
\begin{align*}
\Delta_2 y=&(\theta \otimes \theta )\circ \Delta \circ \omega (y)=-(\theta \otimes \theta )\circ \Delta (\alpha_0)=-y\otimes y.\\
\Delta_2 z=&(\theta \otimes \theta )\circ \Delta \circ \omega (z)=-(\theta \otimes \theta )\circ \Delta (\sum_{i\geq 0}\alpha_i)\\
\end{align*}
\begin{align*}
=&-\Bigl(\theta (\alpha_0 )\otimes \theta (\alpha_0)+\theta (\alpha_0)\otimes \theta (\alpha_1)+\theta (\alpha_1)\otimes \theta (\alpha_0)+\theta (\alpha_1)\otimes \theta (\alpha_1)\Bigr)\\
=&-z\otimes z.\\
\Delta_2 c=&(\theta \otimes \theta )\circ \Delta \circ \omega (c)=(\theta \otimes \theta )\circ \Delta (\sum_{i\geq 0}\frac{1}{i+1}\beta_i)\\
=&{\textstyle(\theta \otimes \theta )\Bigl((\beta_0 \otimes \alpha_0+ \alpha_0\otimes \beta_0)+\frac{1}{2}(\beta_0\otimes \alpha_1+\beta_1\otimes \alpha_0+ \alpha_0\otimes \beta_1+\alpha_1\otimes \beta_0)\Bigr)}\\
=&-\frac{1}{2}c\otimes (y+z)-\frac{1}{2}(y+z)\otimes c
\end{align*}

Next, we show that, for any $k\ge 3$, $\Delta_kz=0$. To do so, we check that $\Delta_Tz=0$ for  any tree $T\in PT_k$.
Since $k\geq 3$, the linear map $\Delta_T$ is of the  form
$$
\Delta_T=\cdots \circ (P\otimes K)\circ \Delta_2\circ \omega,
$$
 where
$P$ is either  $\theta$ or $K$, depending on whether the tree has one or two
internal edges in the first vertex. Hence,
$$
\Delta_Tz=\cdots \circ (P\otimes K)\circ
\Delta_2(-\sum_{j\geq 0}\alpha_j).
$$
Now,
$\Delta_2(\sum_{j\geq 0}\alpha_j)$ is the sum of the following terms:
\begin{equation*}
\xymatrixcolsep{.8pc} \xymatrixrowsep{.8pc}
\entrymodifiers={=<1pc>}
\xymatrix{ &&&&&\alpha_0\otimes \alpha_0&\ar@{..}[dddddlllll]&&&\\&&&&\alpha_1 \otimes \alpha_0&\ar@{..}[ddddrrrr]&\alpha_0\otimes \alpha_1&&&\\
&&&\alpha_2\otimes \alpha_0&\ar@{..}[dddrrr]&\alpha_1 \otimes \alpha_1 &&\alpha_0\otimes \alpha_2&&\\
&&\alpha_3\otimes \alpha_0&\ar@{..}[ddrr]&\alpha_2\otimes \alpha_1&&\alpha_1\otimes
\alpha_2&&\alpha_0\otimes \alpha_3 &\\
&\alpha_4\otimes \alpha_0&&\alpha_3\otimes \alpha_1&&\alpha_2\otimes
\alpha_2&&\alpha_1\otimes\alpha_3&&\alpha_0\otimes \alpha_4\\
&&&&&&&&&
 }
\end{equation*}
Observe  that $P\otimes K$ applied to the left diagonal vanishes since $K(\alpha_0)=0$. The evaluation of $P\otimes K$ on the $i^{th}$ of the  remaining diagonals between dotted lines gives,
\begin{eqnarray*}
\sum_{j\geq 1}P(\alpha_i)\otimes K(\alpha_j)&=&P(\alpha_i)\otimes
K(\alpha_1)+\sum_{j\geq 1}P(\alpha_i)\otimes K(\alpha_{j+1})\\
&=&P(\alpha_i)\otimes \sum_{j\geq 1}\frac{1}{j+1}\beta_j
{-}P(\alpha_i)\otimes \sum_{j\geq 1}\frac{1}{j+1}\beta_j
=0,
\end{eqnarray*}
and therefore $\Delta_Tz=0$.

 Next, observe that
 $$
 (P\otimes K)\circ \Delta_2\circ \omega (y)=-(P\otimes K)(\alpha_0\otimes \alpha_0)=0
  $$
  as $K(\alpha_0)=0$. Thus, $\Delta_Ty=0$ for any tree $T\in PT_k$ and thus, $\Delta_ky=0$ for any $k\ge 3$.

It remains to prove that,
$$\Delta_kc= \sum_{p+q=k-1}\frac{B_{k-1}}{p!q!}c^{\otimes
p}\otimes (y-z)\otimes c^{\otimes q},\ k\geq 3.$$

In our particular situation, since $\omega (c)=\sum_{i\geq
0}\frac{1}{i+1}\beta_i$, $\Delta\,
\beta_j=\sum_{k=0}^j \beta_k\otimes \alpha_{j-k}+\sum_{k=0}^j
\alpha_k\otimes \beta_{j-k}$ and $K(\beta_j)=0$ for $j\geq 0$, the only trees $T$ which contribute to
$\Delta_k \,c$ are those with a terminal edge in each vertex. Otherwise, as two internal edges  with the same vertex $\Delta$ are labeled with
$K$,  the linear map $\Delta_T$ vanishes. Call $\mathcal T\subset PT_k$ the set of these contributing trees which is easily seen to have cardinality $2^{k-2}$. Observe that  each $T\in\mathcal T$ has exactly one vertex in which the two incoming edges are terminal. Moreover, among all the trees in $\mathcal T$, there are $k-1$ different positions in which this special vertex with adjacent terminal edges can be found. In fact, for each $n=1,\dots,k-1$, there are $\binom{k-2}{n-1}$ trees in $\mathcal T$ for which this special vertex is located in the $n^{th}$ position.

For instance, for $k=4$, $\Delta_4c=\sum_{T\in\mathcal T}\pm\Delta_T(c)$ where $\mathcal T$ is the set of the following trees:\vskip .2cm
 {
 \tiny
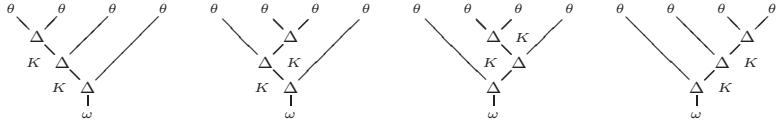
 $$\xymatrixcolsep{.3pc}
\xymatrixrowsep{.3pc} \entrymodifiers={=<.5pc>} \xymatrix{
           &*{^\theta}\ar@{-}[dr]&                      & *{^\theta}\ar@{-}[dl]      &                    &*{^\theta}\ar@{-}[ddll]&&*{^\theta}\ar@{-}[dddlll]&   &{^\theta}\ar@{-}[ddrr]&&{^\theta}\ar@{-}[dr]    &                   &{^\theta}\ar@{-}[dl]&&{^\theta}\ar@{-}[dddlll]&   &{^\theta}\ar@{-}[dddrrr]&&{^\theta}\ar@{-}[dr]&&{^\theta}\ar@{-}[dl]   &                   &{^\theta}\ar@{-}[ddll]  &   &*{^\theta}\ar@{-}[dddrrr]&&*{^\theta}\ar@{-}[ddrr]  &                   &*{^\theta}\ar@{-}[dr]     &                      &*{^\theta}\ar@{-}[dl]  \\
           &                & \Delta\ar@{-}[dr]_K&                       &                    &                  &&                    &   &                 &&                   &\Delta \ar@{-}[dl]^K&               &&                   &   &                   &&               &\Delta\ar@{-}[dr]^K&                   &                   &                 &   &                    &&                    &                   &                     &\Delta\ar@{-}[dl]^K &                 \\
           &                &                      & \Delta\ar@{-}[dr]_K &                    &                  &&                    &   &                 &&\Delta \ar@{-}[dr]_K&                   &               &&                   &   &                   &&               &                   &\Delta\ar@{-}[dl]_K&&                 &   &                    &&                    &                   &\Delta\ar@{-}[ld]^K&                      &                  \\
           &                &                      &                       &\Delta\ar@{-}[d]  &                  &&                    &   &                 &&                   &\Delta \ar@{-}[d] &               &&                   &   &                   &&               &\Delta\ar@{-}[d] &&                   &                 &   &                    &&                    &\Delta\ar@{-}[d] &                     &                      &                  \\
           &                &                      &                       & {_{\stackrel{}{\omega}}}&                  &&                    &   &                 &&                   &{_{\stackrel{}{\omega}}}&               &&                   &   &                   &&               &{_{\stackrel{}{\omega}}}&&                   &                 &   &                    &&                    &{_{\stackrel{}{\omega}}}&                     &                      &
}
$$}

\noindent

Here, there are $3=k-1$ different positions for the special vertex. The first position ($n=1$) appears $1=\binom{2}{0}$ time in the left tree (the upper left $\Delta$).  The second position ($n=2$) appears $2=\binom{2}{1}$ times in the middle two trees (the upper middle $\Delta$). Finally, the third position ($n=3$) appears $1=\binom{2}{2}$ time in the right tree (the upper right $\Delta$).

Next,  using the particular recursive definition of the Bernoulli numbers in $(7)$, a long inductive procedure on $k\ge 3$ shows that, for each  tree $T\in\mathcal T$, with the special vertex  in the $n^{th}$ position, we have
$$\Delta_Tc=\varepsilon_T\frac{B_{k-1}}{(k-1)!}c^{n-1}\otimes \bigl((y-z)\otimes c\otimes + c\otimes(y-z)\bigr)\otimes c^{\otimes k-n-1}.$$

Now,
observe that, for a given $m=1,\dots,k-2$,  the term
 $$
 c^{m}\otimes (y-z)\otimes c^{k-m-1}
 $$
 appears (modulo scalars) in the expression of $\Delta_kc$ for those trees with the special vertex in positions $m$ and $m+1$. For $m=0$ and $m=k-1$,  the expressions $(y-z)\otimes c^{k-1}$ and $c^{k-1}\otimes (y-z)$ appear for those trees with the special vertex in positions $1$ and $k-1$ respectively. In view of this observation we have,
$$
\begin{aligned}\Delta_kc&=\sum_{T\in\mathcal T}\varepsilon_T\Delta_Tc\\
&=\frac{B_{k-1}}{(k-1)!}(y-z)\otimes c^{\otimes k-1}\\
&\quad \sum_{m=1}^{k-2} \frac{B_{k-1}}{(k-1)!}\left(\binom{k-2}{m-1}+\binom{k-2}{m}\right)c^{\otimes m}\otimes (y-z)\otimes c^{\otimes
k-m-1}\\
&\quad\frac{B_{k-1}}{(k-1)!}c^{\otimes k-1}\otimes(y-z)
\\
&=\sum_{m=0}^{k-1}\frac{B_{k-1}}{(k-1)!}\binom{k-1}{m}c^{\otimes
m}\otimes (y-z)\otimes c^{\otimes k-m-1}\\
&=\sum_{p+q=k-1}\frac{B_{k-1}}{p!q!}c^{\otimes
p}\otimes (y-z)\otimes c^{\otimes q}.
\end{aligned}$$

\end{proof}

\section{Homotopy in $\catlinfinito$ via the Lawrence-Sullivan cylinder}

The Lawrence-Sullivan construction $\lasu$ has already proven to be the right cylinder of $S^0$ in the category $\catdgl$ of differential graded Lie algebras \cite[Remark 3.5(2)]{bumu}. We will use it in a similar fashion to introduce  homotopy  in $\catlinfinito$. As one can see in \cite{bumu}, this naturally generalizes both the classical Quillen notion and the homotopy via the gauge action on $\catdgl$.

\begin{definition}
Let $z_0,z_1\in\mc(L)$ with  $L\in\catlinfinito$. We say that $z_0$ is {\em homotopic} to $z_1$ and write $z_0\sim z_1$ if there is a mild $L_\infty$ morphism $\phi\colon \lasu\to L$ such that $\mc(\phi)(a)=z_0$ and $\mc(\phi)(b)=z_1$.
\end{definition}

Given $z\in \mc(L)$ we denote by $\varphi_z\colon \cinfinito(L)\to\bk$ the corresponding augmentation  given by the bijection $\mc(L)\cong \aug\cinfinito(L)$ of Proposition \ref{inicial3}. Explicitly, $\varphi_z=\rho\cinfinito(\phi_u)$ where $\phi_u\colon \lib(u)\to L$ is the  $L_\infty$ morphism  of Lemma \ref{existenciaMC} associated to $z$ and $\rho$ as in Lemma \ref{existenciaaum}. Also, recall that two CDGA morphisms defined over a CDGA, free as commutative graded algebra, $f_0,f_1\colon (\Lambda V,d)\to A$ are said to be homotopic, and write $f_0\sim f_1$ if there is a morphism $\psi\colon  (\Lambda V,d)\to A\otimes\Lambda (t,dt)$ such that $\varepsilon_i\psi=f_i$, $i=0,1$. Here $\varepsilon_i\colon A\otimes\Lambda (t,dt)\to A$ denotes the morphisms extending the identity on $A$ and evaluating $t$ on  $i$.

Our main result in this section reads as follows.

\begin{theorem}\label{homotopia}
Let $z_0,z_1\in\mc(L)$ with  $L\in\catlinfinito$. Then, $z_0\sim z_1$ if and only if $\varphi_{z_0}\sim \varphi_{z_1}$.
\end{theorem}

As one may expect, Theorem \ref{homotopia} reduces to the case of considering the two endpoints of the interval.

\begin{lemma}\label{homotopiaintervalo}
Let $\varphi_a,\varphi_b\colon \cinfinito (\lasu)\to\bk$ be the  augmentations corresponding to the Maurer-Cartan elements $a,b\in \mc(\lasu)$. Then,
$\varphi_a\sim \varphi_b$.
\end{lemma}

\begin{proof} Let $\cinfinito\mathcal{L}(\Theta)\colon \cinfinito (\lasu)\to \cinfinito\mathcal{L}({\mathfrak B})$ be the CDGA morphism obtained from applying $\cinfinito$ to the  DGL quasi-isomorphism $\mathcal{L}(\Theta)\colon \mathcal{L}({\mathfrak B}) \to \lasu $ of Corollary \ref{lie}. Observe that, on generators,
$$
\cinfinito\mathcal{L}(\Theta)=s\Theta^\sharp\colon s\lasu^\sharp\to s\bigl(\lib(s^{-1}{\mathfrak B})\bigr)^\sharp.
$$
In particular,
$$
\cinfinito\mathcal{L}(\Theta)(sa^\sharp)=\alpha_1^\sharp-\alpha_0^\sharp;\quad \cinfinito\mathcal{L}(\Theta)(sb^\sharp)=-\alpha_1^\sharp; \quad \cinfinito\mathcal{L}(\Theta)(sx^\sharp)=\beta_0^\sharp.
$$
Compose this morphism with the natural projection
$$q\colon \cinfinito\mathcal{L}({\mathfrak B})=\Lambda\bigl(s\lib(s^{-1}{\mathfrak B})\bigr)^\sharp\to {\mathfrak B}^\sharp$$
to obtain a CDGA morphism
$$
q\cinfinito\mathcal{L}(\Theta)\colon  \cinfinito (\lasu)\to {\mathfrak B}^\sharp
$$
which, on generators, satisfies the identities above and maps to zero the suspension of any decomposable element in $\lasu$.
Therefore, as $\cinfinito(\lasu)$ is a free commutative algebra and  $q\cinfinito\mathcal{L}(\Theta)$ sends the generators of this algebra to the image of the CDGA inclusion $\Lambda (t,dt)\hookrightarrow {\mathfrak B}^\sharp$ of Remark \ref{dobledual}, $q\cinfinito\mathcal{L}(\Theta)$ factors through $\Lambda (t,dt)$  to provide the CDGA morphism
$$
\Gamma\colon \cinfinito(\lasu)\to \Lambda(t,dt),
$$
$$
\Gamma (sa^\sharp)=t-1,\quad\Gamma(sb^\sharp)= -t,\quad\Gamma(sx^\sharp)=dt,\quad\Gamma(\xi)=0,\,\xi\in\bigl(s\lib^{\ge 2}(a,b,x)\bigr)^\sharp.
$$

On the other hand, observe that  the composite
$$
\sigma\colon \Lambda(t,dt)\hookrightarrow {\mathfrak B}^\sharp\hookrightarrow \cinfinito\mathcal{L}({\mathfrak B})\stackrel{\cinfinito\mathcal{L}(\Omega)}{\longrightarrow}\cinfinito(\lasu),
$$
is a section of $\Gamma$.

Finally, by definition, the CDGA morphism $\varphi_a\colon\cinfinito(\lasu)\to\bk$ (resp. $\varphi_b$) maps $sa^\sharp$ (resp. $sb^\sharp$) to $1$ and any other generator of $\bigl(s\lib(a,b,x)\bigr)^\sharp$ to $0$. Thus, $\varepsilon_0\Gamma=\varphi_a$ and $\varepsilon_1\Gamma=\varphi_b$.
\end{proof}

\begin{proof}[Proof of Theorem  \ref{homotopia}] Assume $z_0\sim z_1$ via the mild $L_\infty$ morphism $\phi\colon \lasu\to L$ and observe that, by Proposition \ref{inicial3}, $\varphi_{z_0}=\varphi_a\cinfinito(\phi)$ and $\varphi_{z_1}=\varphi_b\cinfinito(\phi)$. Thus, by  Lemma \ref{homotopiaintervalo}, $\varphi_{z_i}=\varepsilon_i\Gamma\cinfinito(\phi)$, $i=1,2$, and $\varphi_{z_0}\sim \varphi_{z_1}$ via $\Gamma\cinfinito(\phi)$.

On the other hand if $\varphi_{z_0}\sim \varphi_{z_1}$ via $\psi\colon\cinfinito(L)\to \Lambda(t,dt)$, consider the morphism
$$
\sigma\psi\colon\cinfinito(L)\to\cinfinito(\lasu)
$$
with $\sigma$ the section of $\Gamma$ in the proof of Lemma \ref{homotopiaintervalo}. Write this composition  as $\cinfinito(\phi)$ with $\phi\colon \lasu\to L$ and note, again via Proposition \ref{inicial3}, that $\mc(\phi)(a)=z_0$ and $\mc(\phi)(b)=z_1$.
\end{proof}

On the other hand, one can extend the classical notion of homotopy of Maurer-Cartan elements of DGL's to $\catlinfinito$. For a given $L_\infty$ algebra we denote by $\eta_0,\eta_1\colon L\otimes\Lambda(t,dt)\to L$ the mild $L_\infty$ morphisms resulting from evaluating $t$ in $0$ and $1$ respectively.

\begin{definition}
Let $L\in \catlinfinito$ and let $z_0,z_1\in \mc(L)$. We say that $z_0$ and $z_1$ are {\em Quillen homotopic or simple Q-homotopic} and denote it by $z_0\sim_Qz_1$ if there is a Maurer-Cartan element $\Phi\in\mc \bigl(L\otimes\Lambda (t,dt)\bigr)$ such that $\mc(\eta_i)(\Phi)=z_i$, $i=0,1$.
\end{definition}

Then we have:

\begin{proposition} Let $L\in \catlinfinito$ and $z_0,z_1\in \mc(L)$. Then,
$$z_0\sim z_1\quad\text{if and only if}\quad z_0\sim_Qz_1.
$$
\end{proposition}

\begin{proof}

Assume $z_0\sim_Q z_1$ and let  $\Phi\in\mc \bigl(L\otimes\Lambda (t,dt)\bigr)$ such that $\mc(\eta_i)(\Phi)=z_i$, $i=0,1$. Denote by $\varphi_\Phi\in\aug\cinfinito\bigl(L\otimes\Lambda(t,dt)\bigr)$ and $\varphi_{z_0},\varphi_{z_1}\in\aug\cinfinito(L)$ the associated  augmentations via the bijection in Proposition \ref{inicial3}; see also Remark \ref{finalre}. Then, also in light of Proposition \ref{inicial3}, we have:
$$
\cinfinito(\eta_i)\varphi_\Phi=\varphi_{z_i},\qquad i=0,1.
$$
On the other hand, by Corollary \ref{funtormc}, see also Remark \ref{finalre}, $\varphi_\Phi$ is identified with a  CDGA morphism  $\Psi\colon \cinfinito(L)\to\Lambda(t,dt)$. Via this identification one explicitly checks that the above equation reads
$$
\varepsilon_i\Psi=\varphi_{z_i},\qquad i=0,1,
$$
and, by Theorem \ref{homotopia}, $z_0\sim z_1$. The reverse procedure proves the other implication.
\end{proof}

\begin{remark}\label{gauge}
Through our notion of homotopy via the Lawrence-Sullivan construction one sees immediately the classical connection with homotopy in ${\catdgl}$ via the gauge action. Let  $L$ be either a complete free Lie algebra or any DGL in which the adjoint action of  $L_0$ is locally nilpotent, i.e., for any $x\in L_0$ there is an integer $i$ such that $\ad_x^i=0$. The {\em gauge action} of $L_0$ on $\mc(L)$  (see for instance \cite{lasu,mane}) is defined as follows: given $x\in L_0$ and $z\in \mc(L)$,
$$x*z=e^{\ad_x}(z)-f_x(\partial x),$$
where $e^{\ad_x}=\sum_{n\ge0}\frac{(\ad_x)^n}{n!}$ and, as operator,
$$f_x=\frac{e^{\ad_x}-{\rm id}}{\ad_x}.
$$
Explicitly,
$$
x*z=\sum_{i\ge0}\frac{\ad_x^i(z)}{i!}-\sum_{i\ge0}\frac{\ad_x^{i}(\partial x)}{(i+1)!}.
$$
We say that two elements $z_0,z_1\in \mc(L)$ are {\em gauge homotopic} and write $z_0\sim_g z_1$ if $x*z_0=z_1$ for some $x\in L_0$. Then it is easy to see \cite[Proposition 3.1]{bumu} that $z_0\sim_gz_1$ if and only if there is a DGL morphism $\phi\colon\lasu\to L$ such that $\Phi(a)=z_0$ and $\Phi(b)=z_1$. \end{remark}

\begin{definition}\label{moduli}  Denote by $\widetilde\mc(L)=\mc(L)/\sim$ the quotient set of homotopy classes of Maurer-Cartan elements. Whenever there is no ambiguity we do not distinguish between a Maurer-Cartan element and its class in  $\widetilde\mc(L)$.
\end{definition}

An immediate consequence of Theorem \ref{homotopia} reads as follows.

\begin{proposition} Let $z_0,z_1\in \mc(L)$, with $z_0\sim z_1$, and let $g\colon L\to M$ be a mild $L_\infty$ morphism, then $\mc(g)(z_0)\sim\mc(g)(z_1)$. In particular $\mc(g)$ induces a map
$$
\widetilde\mc(g)\colon\widetilde\mc(L)\to\widetilde\mc(M)
$$
\end{proposition}

\begin{proof}  Let $\varphi_i\colon \cinfinito(L)\to\bk$ be the augmentations representing $z_i$, $i=0,1$. Then, by Theorem \ref{homotopia} $\varphi_0\sim\varphi_1$. On the other hand,
 Proposition \ref{inicial3} identifies $\mc(g)$  with the map
$$
\aug\cinfinito(g)\colon \aug \cinfinito(L)\to\aug\cinfinito(M)
$$
defined by composition. Then,
$$
\mc(L)(z_0)=\cinfinito(g)\varphi_{0}\sim \cinfinito(g)\varphi_{1}=\mc(L)(z_1).
$$
\end{proof}

However,
 It has been known for a long time that quasi-isomorphisms of $L_\infty$ algebras do not preserve, in general, homotopy classes of Maurer-Cartan elements \cite{kon} unless special restrictions are required (see for instance \cite[Proposition 4.9]{getz}). That is, $\widetilde\mc(g)$ is not one-to-one in general for a given quasi-isomorphism.

 An example of this, which also shows that the cochain functor does not preserve quasi-isomorphism is the following: the  DGL  $(\lib(u),\partial)$ generated by a Maurer-Cartan element is trivially quasi-isomorphic to $0$ whose Maurer-Cartan set is empty. On the other hand, see the beginning of Section 2, the cohomology algebra of $\cinfinito(\lib(u),\partial)$ is isomorphic to two copies of the field in degree zero while the cochain algebra of the trivial $L_\infty$ algebra is isomorphic to the ground field.

 In order to keep invariance of the set of homotopy classes of Maurer-Cartan elements we need a stronger version of quasi-isomorphisms. This is precisely the aim of the next proposition which also extends \cite[Theorem 3.4.3]{ke} to $\catlinfinito$.

\begin{proposition} Let $g\colon L\to M$ be a morphism in $\catlinfinito$ such that $\cinfinito(g)$ is a quasi-isomorphism of cofibrant CDGA's. Then,
$$
\widetilde\mc(g)\colon \widetilde\mc(L)\stackrel{\cong}{\longrightarrow}\widetilde\mc(M)
$$
is a bijection.
\end{proposition}

\begin{proof} Write $\cinfinito(g)$ as $\psi\colon(\Lambda W,d)\stackrel{\simeq}{\to} (\Lambda V,d)$. The fact that $\psi$ induces a bijection on homotopy classes of augmentations
$$
Aug(\Lambda W,d)/\sim\,\,\,\, \stackrel{\cong}{\longrightarrow}\,\, Aug(\Lambda V,d)/\sim
$$
is a classical and well known fact on $\catcdga$ easily deduced via the homotopy lifting of morphisms between cofibrant objects.
\end{proof}

As this proposition shows, the right thing to do, at least from the axiomatic point of view, would be to consider  the stronger class of quasi-isomorphisms formed by those $g$ for which $\mathcal{C}_\infty (g)$ or equivalently $\cinfinito(g)$ are quasi-isomorphisms. Indeed, this is the dual of the choice of Hinich in \cite[Theorem 3.1]{hi2} to successfully endow the category of cocommutative differential graded coalgebras with a structure of closed model category.

We finish this section by briefly fitting the homotopy notion into an axiomatic framework. It is known, see \cite[\S4]{bousgu} for non-negatively graded CDGA's and \cite[\S4]{hi} for the general case,  that $\catcdga$ has a structure of closed model category in which the fibrations are surjective morphisms and the weak equivalences are the quasi-isomorphisms. Cofibrations are the maps which have the lifting property with respect to acyclic fibrations and are characterized as retracts of the so called standard cofibrations. A path object for this closed model structure is the acyclic algebra $\Lambda(t,dt)$  and thus,  the usual homotopy  on morphisms departing from a cofibrant CDGA is an equivalence relation \cite[Lemma 4]{qui2}. In particular, the  homotopy  notion on $\mc(L)$ for a given  $L_\infty$ algebra $L$ is an equivalence relation as long as $\cinfinito(L)$ is a cofibrant CDGA, which is not always the case. Indeed \cite[Remark 2.2.5]{hi}, cofibrant algebras are the $\bz$-graded version of the classical {\em Sullivan algebras} \cite[\S12]{fehatho}. Hence, not all free CDGA's are cofibrant. A classical and easy counterexample is given by the free CDGA $A=(\Lambda(x,y,z),d)$ generated by degree $1$ elements in which $dx=yz$, $dy=zx$ and $dz=xy$. Consider the CDGA $B=\langle a,b,c\rangle $ generated by degree $1$ elements and with trivial products and differential. Let $\beta\colon (\Lambda (u,v,w,\ldots),d)\stackrel{\simeq}{\to} B$  be its minimal model, which sends the cycles $u,v,w$ to $a,b,c$ respectively.  Then, the CDGA morphism $\gamma\colon A\to B$, which sends $x,y,z$ to $a,b,c$ respectively, does not have a lifting to $\beta$ which is a trivial fibration. Even if one considers the slightly different closed model structure on $\catcdga$, arising from dualizing the one in the category ${\bf CDGC}$ of cocommutative differential graded coalgebras \cite[\S3]{hi2}, $A$ is not cofibrant as $\beta$ is again a surjective weak equivalence with this closed model structure \cite[Proposition 3.3.2(3)]{hi2}.

Thus, whenever one wants to use the full potential of the closed model category structure in $\catcdga$ to derive geometrical properties of a given  $L\in\catlinfinito$ via its cochain functor, it is necessary to assume $\cinfinito(L)$ to be a cofibrant CDGA.  This is equivalent to any of the following.

(i) There exists a well ordered basis $\{x_i\}_{i\in I}$ of $L$   such that, for each $k\ge 1$, the class of
$$
\ell_k(x_{i_1},\ldots,x_{i_k})
$$
is zero in $L/L^{>j}$
where $j=\max\{i_1,\ldots,i_k\}$ and $L^{<j}$ stands for the span of $\{x_i\}_{i<j}$.

(ii) Denote by $G^r$ the subspace of $L$ generated by the image of the maps obtained by composition of at most $r$ operations in  $\{\ell_k\}$. Note that $L=G^0\supset G^1\supset\cdots\supset G^r\supset G^{r+1}\supset\cdots$.  Then, for each $x\in L$, there exists some $r$ such that $x\in G^r$ and $[x]\not=0$ in  $G^r/G^{r+1}$.

\section{Realization, components, homotopy  invariance}\label{casifinal}

Points of a given space, or better, $0$-simplices of a  simplicial set are characterized by their algebraic counterparts: they are augmentations of a  CDGA or Maurer-Cartan elements of a DGL, or of an $L_\infty$ algebra, with each of these objects modeling the given simplicial set.  Moreover, one can recover the path component $S_x$ of  a $0$-simplex $x$ in a simplicial set $S$, by carefully truncating the CDGA, DGL or $L_\infty$ algebra modeling $S$,  and perturbing the differential via the augmentation or Maurer-Cartan element describing the given $0$-simplex. This goes back to \cite{brown,haef} in $\catcdga$, to \cite{bufemu} in $\catdgl$, and to \cite{ber,bufemu2} for $L_\infty$ algebras. Here, we set the connection between these procedures and prove their homotopy invariance both in $\catcdga$ and $\catlinfinito$.

In the $\catcdga$ setting we follow the approach and notation in \cite[\S4]{bumu2}.   Let   $f\colon A\to\bk$ be an augmentation of the CDGA $A$ which is, by definition, a $0$-simplex of the Sullivan realization $\langle
A \rangle$. Consider the differential ideal $K_f$ of $A$ generated by
 $A^{< 0}$, $d A^0$ and $\{ a -f(a ),\, a \in A^0 \}$. Then, when $A$ is a
free CDGA $(\Lambda V,d)$, and $f$ is an augmentation, the proof of \cite[Theorem 6.1]{brown}  shows that the projection $\Lambda
V\to \Lambda V/K_f$ induces a homotopy  equivalence of simplicial sets
$$\langle \Lambda V/K_f\rangle \stackrel{\simeq }{\to }\langle
\Lambda V\rangle_f.$$

 One of the main result of this section is that free CDGA's preserve the homotopy type when localized at homotopy augmentations.

\begin{theorem}\label{invarianzadga}
Let $f,g\colon(\Lambda V,d)\to\bk$ be homotopic augmentations. Then, the CDGA's $\Lambda V/K_f$ and $\Lambda V/K_g$ have the same homotopy type. In particular, as both are positively graded, the simplicial sets $\langle
\Lambda V\rangle_f$ and $\langle
\Lambda V\rangle_g$ are homotopy equivalent.
\end{theorem}

\begin{proof}
Let $h\colon \Lambda V\to \Lambda(t,dt)$ be a homotopy between $f$ an $g$. It induces a natural linear map $H\colon\Lambda V\otimes {\mathfrak B}\to \bk$, where ${\mathfrak B}$ is the universal acyclic CDGC (see Definition \ref{universal2}), given by $H(\Phi\otimes\gamma)= (-1)^{|\Phi||\gamma|}  \gamma\bigl(h(\Phi)\bigr)$. We extend this map to a CDGA morphism by considering the  Lannes functor \cite{lannes}  in the category $\catcdga$ (see \cite[\S3]{brown} or \cite[\S2]{bumu2}). Let $\Lambda(\Lambda V\otimes {\mathfrak B})$ be the free algebra generated by $\Lambda V\otimes {\mathfrak B}$  with the differential induced by the tensor product differential on the generators. Consider the ideal $J$ generated by $1\otimes \alpha_0-1$ and by the elements of the form
$$uv\otimes \gamma -\sum_j (-1)^{|v| |{\gamma_j}' |}(u\otimes {\gamma_j}'
)(v\otimes {\gamma_j}'' ),$$
 where $u,v\in V,\ \gamma \in {\mathfrak B}$, and $ \Delta\gamma =\sum_j
{\gamma_j}'\otimes {\gamma_j}''$. Then, the composition
$$\Lambda (V\otimes {\mathfrak B})\subset   \Lambda (\Lambda V\otimes {\mathfrak B})\twoheadrightarrow  \Lambda (\Lambda V\otimes
{\mathfrak B})/J$$
is an isomorphism of graded algebras \cite[Theorem 1.2]{brown}, and we consider in $\Lambda (V\otimes {\mathfrak B})$ the
differential
$\widetilde d$ so that it becomes  a CDGA isomorphism. Explicitly, set $\Delta^{(0)}= {\rm id}_{\mathfrak B}$, $\Delta^{(1)}= \Delta$ and
$$
\Delta^{(m)}=(\Delta \otimes  {\rm id}_{\mathfrak B} \otimes \cdots \otimes {\rm id}_{\mathfrak B}) \circ \Delta^{(m-1)} \colon {\mathfrak B}\to {\mathfrak B}\otimes\stackrel{m+1}{\dots}\otimes {\mathfrak B}.$$
Then, for any element $v\in V$ with $dv=\sum_i v_i^1\cdots v_i^m$, and any $\gamma \in {\mathfrak B}$ with $\Delta^{(m-1)}\gamma=\sum_j \gamma_{j}^1\otimes \cdots \otimes \gamma_{j}^m$,
$$\widetilde{d} (v\otimes \gamma )=\sum_{i,j} (-1)^{\varepsilon} (v_i^1\otimes \gamma_{j}^1)\cdots (v_i^m\otimes \gamma_{j}^m) + (-1)^{|v|}v\otimes \delta\gamma,$$
where  $\varepsilon$ is the sign provided by the Koszul convention.

Next observe that the CDGA morphism $\Lambda H\colon\Lambda(\Lambda V\otimes {\mathfrak B})\to\bk$ maps $J$ to $0$ so it induces an augmentation, denoted in the same way  for simplicity on the notation,
$$
H\colon(\Lambda (V\otimes {\mathfrak B}),\widetilde d)\to\bk,\qquad H(v\otimes\gamma)=(-1)^{|v|}  \gamma\bigl(h(v)\bigr).
$$
Now consider the injective algebra morphisms,
$$
\Lambda V\stackrel{\varphi}{\hookrightarrow}\Lambda(V\otimes {\mathfrak B}) \stackrel{\psi}{\hookleftarrow} \Lambda V,\quad
\varphi(v)=v\otimes\alpha_0,\quad\psi(v)=v\otimes\sum_{i\ge0} \alpha_i,
$$
and observe that, since $\alpha_0$ and $\sum_{i\geq 0} \alpha_i$ are counits of the coalgebra ${\mathfrak B}$, the morphisms $\varphi $ and $\psi $ commutes with differentials  by the explicit definition of $\widetilde{d}$ given above.

 On the other hand,
as $h$ is a homotopy from $f$ to $g$, given $v\in V^0$, $h(v)=\sum_{j=0}^n\lambda_jt^j$ with $\lambda_0=f(v)$ and $\sum_{j=0}^n\lambda_j=g(v)$. Then, one easily sees that $H\varphi=f$ and $H\psi=g$, that is $\varphi,\psi$ are morphisms of augmented CDGA's. In particular, $\varphi(K_f),\psi(K_g)\subset K_H$, and we have induced CDGA morphisms:
$$
\Lambda V/K_f\stackrel{\overline\varphi}{\longrightarrow}\Lambda(V\otimes {\mathfrak B})/K_H \stackrel{\overline\psi}{\longleftarrow} \Lambda V/K_g.
$$
We prove the theorem by showing that both $\overline\varphi$ and $\overline\psi$ are quasi-isomorphisms.
The linear parts of $\varphi$ and $\psi$ are
$$
(V,d_1)\stackrel{\varphi_1}{\longrightarrow}(V,d_1)\otimes ({\mathfrak B},\delta)\stackrel{\psi_1}{\longleftarrow}(V,d_1),\quad \varphi_1(v)=v\otimes\alpha_0,\quad\psi_1(v)=v\otimes\sum_{i\ge0} \alpha_i.
$$
As  $H_*({\mathfrak B},\delta)=H_0({\mathfrak B},\delta)\cong \bk$, and either $\alpha_0$ or $\sum_{i\ge0} \alpha_i$ are cycles representing the only non vanishing homology class, both maps above are quasi-isomorphisms so are $\varphi$ and $\psi$.

 On the other hand, write ${\mathfrak B}=\bk\alpha_0\oplus C\oplus \delta C$ and observe that, if we define $U$ as a copy of the graded vector space $V\otimes C$, the map
  $$
\eta\colon (\Lambda (V\otimes\bk\alpha_0),\widetilde d)\otimes \Lambda (U\otimes\widetilde d U)\stackrel{\cong}{\longrightarrow} (\Lambda(V\otimes {\mathfrak B}),\widetilde d)
$$
which is the inclusion on $\Lambda (V\otimes\bk\alpha_0)$, and sends $U$ to its isomorphic copy $V\otimes C$, is an isomorphism of CDGA's. Moreover, the ideal $K_{\eta H}$ corresponding to the augmentation  $\eta H$ is the sum of differential ideals,
$$
K_{\eta H}=K_{f'}\otimes \Lambda (U\otimes\widetilde d U)+ \Lambda (V\otimes\bk\alpha_0)\otimes K_u$$
where
$$
\begin{aligned}
&f'\colon(\Lambda (V\otimes\bk\alpha_0)\hookrightarrow  (\Lambda (V\otimes\bk\alpha_0),\widetilde d)\otimes \Lambda (U\otimes\widetilde d U)\stackrel{\eta H}{\longrightarrow}\bk,\\
&u\colon \Lambda (U\otimes\widetilde d U)\hookrightarrow  (\Lambda (V\otimes\bk\alpha_0),\widetilde d)\otimes \Lambda (U\otimes\widetilde d U)\stackrel{\eta H}{\longrightarrow}\bk.
\end{aligned}
$$
Obviously, $\eta$ sends $K_{\eta H}$ isomorphically to $K_H$. Moreover, it sends $K_{f'}\otimes 1$ to the image of $\varphi(K_f)$ while $K_u$ and thus
$\Lambda (V\otimes\bk\alpha_0)\otimes K_u$ is acyclic. Therefore, the inclusion
$\varphi\colon K_f\stackrel{\simeq}{\hookrightarrow}K_H$
 is a quasi-isomorphism. Exactly the same procedure shows that $\psi\colon K_g\stackrel{\simeq}{\hookrightarrow}K_H$ is also a quasi-isomorphism.

 Finally the Five Lemma establishes that $\overline\varphi$ and $\overline\psi$ are quasi-isomorphisms.
 \end{proof}

 We now turn to $\catlinfinito$. Let $L$ be an $L_\infty$ algebra, let $z\in
\MC(L)$ and consider the perturbed $L_\infty$ algebra $(L^z,\{\ell^z_k\})$ defined
in Section 1. Note that the Maurer-Cartan sets of $L$ and $L^z$ are bijective as   $\mc(L^z)=\{a-z,\,a\in\mc(L)\}$ \cite[Lemma 4.8]{ber}.

On the other hand, at the sight of the fact that,  not every reordering of the vertices of a simplicial set gives rise to a homeomorphism, one should not expect that, for homotopic Maurer-Cartan elements $z_0\sim z_1$, the  $L_\infty$ algebras, $L^{z_0}$ and $L^{z_1}$ be isomorphic in general.
We will prove, however, that the geometrical realization of perturbations of an $L_\infty$ algebra by homotopic Maurer-Cartan elements do have the same homotopy type.

Given $L\in \catlinfinito$  and $z\in\mc(L)$, truncate $L^z$ to produce a non-negatively graded  $L_\infty$ algebra
$L^{(z)}$ (the notation coming from classical localization)  whose underlying graded vector space is
$$L^{(z)}_i=\left\lbrace
           \begin{array}{c l}
              L_i^z=L_i& \text{if } i> 0,\\
              \ker \ell_1^{z}& \text{if } i=0,\\
              0 & \text{if } i<0,
           \end{array}
         \right.$$
and with  brackets induced by $\ell^z_k$ for any $k\ge 1$.
We first easily recover a known result.

\begin{theorem}{\em \cite[Corollary 1.2]{ber} \cite[Theorem 1.1]{bufemu2}}  Let $\varphi\colon \cinfinito(L)\to\bk$ be the augmentation corresponding to the Maurer-Cartan element $z$ of a given mild $L_\infty$ algebra. Then $\langle L\rangle_\varphi$ and $\langle L^{(z)}\rangle$ are homotopy equivalent simplicial sets.
\end{theorem}

\begin{proof}
First, observe that, for a given augmentation $f\colon(\Lambda V,d)\to \bk$ of a free CDGA, the quotient  $(\Lambda V, d )/K_f$ is again a  free CDGA $(\Lambda
(\overline{V}^1\oplus V^{\geq 2}),d_f)$ in which $\overline{V}^1$ is the coker of the map $\overline d\colon V^0\to  V^1$ resulting by applying the differential $d$ and then projecting over the ideal generated by $V^{< 0}$ and     $\{ v -f(v ),\, v \in V^0 \}$.

In particular, if $(\Lambda V,d)=\cinfinito(L)$, we write,
$$
\cinfinito(L)/K_{\varphi}=(\Lambda (\overline{V}^1\oplus
V^{\geq 2}),d_{\varphi}) .
$$
A straightforward computation shows that
$
(\Lambda (\overline{V}^1\oplus
V^{\geq 2}),d_{\varphi}) $ is precisely ${C}^\infty(L^{(z)})$. Then,
$$\langle L^{(z)}\rangle=\langle {C}^\infty(L^{(z)})\rangle=\langle\cinfinito(L)/K_{\varphi}\rangle\simeq \langle\cinfinito(L)\rangle_{\varphi}=\langle L\rangle_\varphi.$$
\end{proof}

\begin{remark}
With the notation in Theorem above, we may choose the isomorphism of (non differential!) algebras
$$
\gamma_f\colon\Lambda V\stackrel{\cong}{\longrightarrow}\Lambda V,\qquad\gamma_f(v)=v-f(v),\quad v\in V,
$$
and consider the differential $d_f=\gamma_fd\gamma_f^{-1}$ so that $\gamma_f\colon (\Lambda V,d)\stackrel{\cong}{\to}(\Lambda V,d_f)$ becomes a CDGA isomorphism. Geometrically, this only means that we are considering a different base point. Indeed,   $d_f$  has no scalar terms and thus it defines an $L_\infty$ structure on $L$ if and only if $f$ is an augmentation. We also note that the exponential does not preserves augmentations. For instance, consider the non base point of $S^0$, i.e., the augmentation of the CDGA $\varepsilon\colon (\Lambda (x,y),d)\to \bk$, $\varepsilon(x)=1$. Recall that $x$ is a cycle of degree zero and $dy=\frac{1}{2}(x^2-x)$. Again, one can consider the non differential automorphism
$$
\gamma_{e^\varepsilon}\colon\Lambda (x,y)\stackrel{\cong}{\longrightarrow}\Lambda (x,y), \quad \gamma_{e^\varepsilon}(x)=x+\sum_{n\ge 1}\frac{1}{n!}\varepsilon^n(x)=x+e.
$$
Besides the obvious restriction of working with a field $\bk$ containing the real number $e$, one sees that in $(\Lambda (x,y),d_{e^\varepsilon})$, $d_{e^\varepsilon}y$ contains the scalar $\frac{1}{2}(e^2-e)$ and thus, this CDGA is not the chain algebra of any $L_\infty$ algebra.
\end{remark}

An immediate consequence of theorems \ref{homotopia} and  \ref{invarianzadga} reads as follows.

\begin{corollary} \label{invarianzalie} Let $z_0,z_1\in \mc(L)$ with $z_0\sim z_1$. Then, the  simplicial sets $\langle L^{(z_0)}\rangle$ and $\langle L^{(z_1)}\rangle$ are homotopy equivalent.
\end{corollary}

\begin{proof}  If $\varphi_{z_0},\varphi_{z_0}\colon\cinfinito(L)\to\bk$ are the augmentations associated to $z_0$ and $z_1$ respectively, Theorem \ref{homotopia} asserts that $\varphi_{z_0}\sim\varphi_{z_0}$. Then, by Theorem \ref{invarianzadga}, $\langle L\rangle_{\varphi_{z_0}}\simeq \langle L\rangle_{\varphi_{z_1}}$ which, in view of the theorem above is equivalent to stating that $\langle L^{(z_0)}\rangle\simeq\langle L^{(z_1)}\rangle$.
\end{proof}

We may summarize the results in this section by the following.

\begin{theorem}\label{realizacion}
Let $L$ be a mild $L_\infty$ algebra. Then,
$$
\langle L\rangle\simeq \,\stackrel{\cdot}{\cup}_{z\in \widetilde\mc(L)}\langle L^{(z)}\rangle.\eqno{\square}$$
\end{theorem}

\begin{example}\label{ejemplo}
 The realization of the Lawrence-Sullivan construction  has the homotopy type of $S^0$. Indeed, $\lasu$ has two non homotopic Maurer-Cartan elements $\{0,a\}$ as $a$ and $b$ are gauge homotopic via $x$ (see Remark \ref{gauge}). In both cases $\lasu^{(0)}=\lasu^{(a)}=0$. Hence,
$\langle \lasu\rangle\simeq\langle 0\rangle\stackrel{\cdot}{\cup}\langle 0\rangle\simeq S^0$. On the other hand, if we consider the free DGL $\lib(b)$ generated by the Maurer-Cartan element $b$, the same computation shows that also $\langle \lib(b)\rangle\simeq S^0$. Nevertheless, from a functorial point of view, $\lasu$ is known to be a cylinder of $\lib(b)$ \cite[\S3]{bumu}. As in the based homotopy category, the cylinder of $S^0$ is  $S^0\wedge
I=I^+$  the disjoint union of the interval with an exterior point, it is more accurate to state that
 $$
 \langle\lasu\rangle\simeq I^+.
 $$

 On the other hand, the realization of the inclusion $k\colon \lib(b)\hookrightarrow\lasu$ is, up to homotopy, the based map $S^0\to I^+$ which sends the non base point of $S^0$ to any of the endpoints of the interval. The cofibre of this map is the interval $I$ while the algebraic cofibre of $k$ is  $
\lasu_I=(\widehat{\mathbb{L}}(a,x),\partial)$ where $a$ is a Maurer-Cartan element and $\dlie(x)=-\sum_{i\ge 0}\frac{B_i}{i!}{\rm ad}^i_x(a)$. Observe that $\langle\lasu_I\rangle$ is contractible, as its two Maurer-Cartan elements $0$ and $a$ are homotopic, and thus, we may think of $\lasu_I$ as a model for the interval. This model has been used in \cite[\S4]{bumu6} to find a natural $L_\infty$ model of the based path space.

\end{example}

\section{Algebraic models of  non-connected spaces}

As a result of past sections, we are able to develop here a  procedure  to obtain DGL and CDGA models of non-connected spaces. The initial data is a non necessarily path connected space whose path connected components are nilpotent spaces of the homotopy type of finite type  CW-complexes. In this section all path connected spaces will be of this kind. Equivalently, we may start with  a family of CDGA's or DGL's, modeling in the Sullivan or Quillen sense respectively, each of the path components of the given space. Our  assumptions  let us choose either Quillen models  or  the Quillen functor of CDGA models of the components of the spaces, as they are homotopy equivalent \cite{ma}. We then glue them together to obtain  a  DGL  whose decomposition, following localization at each homotopy class of Maurer-Cartan elements, gives precisely, and up to homotopy, the given  DGL family. In other words, its realization  has the homotopy type of the rationalization of the space under consideration.

Given $L,M\in\catdgl$ we denote by $L*M$ its coproduct. Recall that, given free presentations $L=\lib(U)/I$, $M=\lib(V)/J$, then $L*M=\lib(U\oplus V)/\langle I,J\rangle$.

\begin{lemma}\label{lielema} Let $L$ be a DGL and let $\lib(u)$ be the DGL generated by the Maurer-Cartan element $u$. Then,  $(\lib(u)*L)^{u}$ is an acyclic DGL.
\end{lemma}

\begin{proof} Observe that $(\lib(u)*L)^{u}=(\lib(u)*L,\partial_u)$ in which $\partial_u(u)=\frac{1}{2}[u,u]$ and $\partial_u(x)=\partial(x)+[u,x]$ for $x\in L$ with $\partial$ the differential in $L$. Let
$$
(\lib(u)*L)^{u}=I_0\supset\dots\supset I_p\supset I_{p+1}\supset\dots
$$
be the decreasing sequence of differential ideals in which,
for each $p\ge 0$,  $I_p=\im {\ad_u^p}_{|_L}$, i.e., it  is generated by
$$\{\bigl[\underbrace{u,[u,\dots ,[u}_{p}, x]\bigr]\dots \bigr],\, x\in L\}.$$
The zero term of  the resulting spectral sequence  is
$
(E^0,d^0)=(\lib(u)*L,0*\partial)$ while $E^1=\lib(u)*H(L)$ with $d^1=\partial_u$ on $\lib(u)$ and $d^1=\ad_u$ on $H(L)$.

 A straightforward computation shows that $H(E^1,d^1)=0$ and thus $(\lib(u)*L)^{u}$ is acyclic.
\end{proof}

 \begin{proposition}\label{xpunteado} Let $L$ be a non-negatively graded DGL model of a path connected space $X$. then $\lib(u)*L$ is a model of $X^+=X\stackrel{\cdot}{\cup}{*}$, the disjoint union of $X$ and an exterior point.
 \end{proposition}

 \begin{proof}
 Since $L$ is non-negatively graded, $\widetilde\mc(\lib(u)*L)=\{0,u\}$. Observe that, as homology preserves coproducts and $\lib(u)$ is acyclic, $H(\lib(u)*L)\cong H(\lib(u))*H(L)\cong H(L)$. In particular $H\bigl((\lib(u)*L)^{(0)})=H_{\ge 0}(\lib(u)*L)\cong H(L)$ and the natural inclusion $L\stackrel{\simeq}{\hookrightarrow} (\lib(u)*L)^{(0)}$ is a quasi-isomorphism between non-negatively graded DGL's. This shows that $(\lib(u)*L)^{(0)}$ is a model of $X$. On the other hand, by Lemma \ref{lielema} above,  $(\lib(u)*L)^{u}$ and thus $(\lib(u)*L)^{(u)}$ is acyclic. Theorem \ref{realizacion} finishes the proof.
  \end{proof}

 Observe that $\lib(u)*L$  mimics in the algebraic setting the wedge of $S^0$ (see example \ref{ejemplo}) and $X$ which yields $X^+$. However, in these constructions, the  base points are $0\in \lib(u)*L$ and the one in $X$ for $X^+$. We need to rearrange this if we want to describe  in the DGL setting the disjoint union $X\stackrel{\cdot}{\cup}Y$ as the wedge $X^+\vee Y$. Here, the base point of $X^+$ is now the exterior point. For that, we first observe that
perturbing a DGL does not affect the homotopy type of its geometric realization. It only rearranges the basepoints of its components.

\begin{proposition}
Let $L$ be a DGL  and $z\in\mc(L)$. Then,
$$
\langle L^z\rangle\simeq\langle L\rangle.
$$
\end{proposition}

\begin{proof}
This is a direct consequence of the following facts. On the one hand, and for any $L_\infty$ algebra, it is easy to check \cite[Lemma 4.8]{ber} that $\mc(L^z)=\{a-z,\,a\in\mc(L)\}$. On the other hand, a trivial computation shows that, whenever $L$ is a DGL, $(L^z)^{a-z}=L^a$. The proposition follows at once from Theorem \ref{realizacion}.
\end{proof}

With this in mind, let $X$ be a space with path components $\{Y,X_j\}_{j\in J}$ and let $\{L,L_j\}_{j\in J}$ be a family of non-negatively graded DGL's, each of which  modeling the corresponding component. For each $j\in J$ consider the perturbed DGL $M_j=(\lib(u_j)*L_j)^{u_j}$. That is, $M_j=(\lib(u_j)*L_j,\partial_{u_j})$ in which $\partial_{u_j}(u_j)=\frac{1}{2}[u_j,u_j]$ and $\partial_{u_j}x=\partial_jx+[u_j,x]$, with  $x\in L_j$ and $\partial_j$ the differential in $L_j$. Then we have the following.

\begin{theorem}\label{modelo} The DGL $
M=*_{j\in J}M_j*L
$
is a model of $X$.
\end{theorem}
An immediate consequence reads as follows.

\begin{corollary}
$\cinfinito(M)$ is a free CDGA model of X.\hfill$\square$
\end{corollary}

\begin{proof}[Proof of Theorem \ref{modelo}] Clearly, in light  of the proposition above and its proof, $$\widetilde\mc(M)=\{0\}\cup\{-u_j\}_{j\in J}.$$
 We prove that $M^{(0)}$ and each $M^{(-u_j)}$ are non-negatively graded DGL's of the same homotopy type of $Y$ and $X_j$ respectively. By Lemma \ref{lielema} each $M_j$ is acyclic. Thus $H(M^{(0)})=H_{\ge 0}(M)\cong H(L)$ and the inclusion $L\stackrel{\simeq}{\hookrightarrow}M^{(0)}$ is a quasi-isomorphism.

On the other hand, an easy computation shows that
$$
M^{-u_j}=(\lib(u_j) * N)^{u_j} * L_j,\qquad\text{with}\qquad N=*_{i\not=j} M_j* L.
$$
Again by Lemma \ref{lielema}, $(\lib(u_j)*N)^{u_j}$ is acyclic and the inclusion $L_j\stackrel{\simeq}{\hookrightarrow}M^{(-u_j)}$ is  a quasi-isomorphism. Theorem \ref{realizacion} finishes the proof.
\end{proof}

\begin{remark} Observe that for any  $L\in\catdgl$, we may consider its decomposition via localization, $\{L^{(z)}\}_{z\in\widetilde\mc(L)}$, and then obtain  $M\in\catdgl$ as in the theorem above. Hence, $\langle L\rangle\simeq\langle M\rangle$. That is, in geometric terms, L is homotopy equivalent to a free (not bounded below!) DGL generated by a vector space concentrated in degrees greater than or equal to $-1$ and with a minimum set of Maurer-Cartan elements, $\widetilde\mc(L)=\mc(M)$.
\end{remark}

\begin{example} (1) {\em Model of a disjoint union of spheres.} Let $X=\stackrel{\cdot}{\cup}_{i\in I}S^{n_i}$ be the disjoint union of spheres with $n_i\ge 1$. Fix $i_0\in I$ and observe that, by Theorem \ref{modelo}, a DGL model of $X$ is
$
(\lib(W),\partial)$ in which $W$ is generated by $\{u_i\}_{i\not=i_0}$ and $\{a_i\}_{i\in I}$ with $|u_i|=-1$, $|a_i|=n_i-1$, and the differential is given by $\partial u_i=\frac{1}{2}[u_i,u_i]$, $\partial a_{i_0}=0$ and $\partial a_i=[a_i,u_i]$ if $i\not=i_0$.

(2) We invite the reader to check that applying Proposition \ref{xpunteado} to  the model for the interval $\lasu_I=(\widehat{\mathbb{L}}(a,z),\partial)$  given in Example \ref{ejemplo}, we recover precisely $\lasu$.
\end{example}

\bigskip
\bigskip
\bigskip

\noindent{\sc Institut de Math\'ematique, Universit\'e Catholique de Louvain , 2 Chemin du Cyclotron B-1348, Louvain-la-Neuve, Belgium}.\hfill\break
{urtzibuijs@gmail.com}

\bigskip

\noindent {\sc Departamento de \'Algebra, Geometr\'{\i}a y Topolog\'{\i}a, Universidad de M\'alaga, Ap.\ 59, 29080 M\'alaga, Spain}.\hfill\break {aniceto@uma.es}


\begin{thebibliography}{99}

\bibitem{ber2} A. Berglund, $A_\infty$-algebras and homological perturbation theory,
\textit{Preprint}, (2010).


\bibitem{ber} A. Berglund, Rational homotopy theory of
mapping spaces via Lie theory for $L_\infty $-algebras,
\textit{Preprint} http://arxiv.org/abs/1110.6145.

\bibitem{bousgu} A. Bousfield and V. Guggenheim, On PL De Rham Theory and rational homotopy type, {\em Mem. Amer. Math. Soc.}, {\bf 179} (1976).


\bibitem{brown} E. H. Brown \and R. H. Szczarba, On the rational homotopy type of function spaces, \textit{Trans. Amer. Math. Soc.},
 {\bf 349} (1997), 4931--4951.
 
 \bibitem{bufemu} U. Buijs, Y. F\' elix \and A. Murillo, Lie models
for the components of sections of a nilpotent fibration, {\em
Trans. Amer. Math. Soc. } {\bf 361}(10) (2009), 5601--5614.


\bibitem{bufemu2} U. Buijs, Y. F\'elix and A. Murillo, $L_\infty$ Rational
homotopy of mapping spaces,  {\em Rev. Mat. Complutense}, DOI: 10.1007/s13163-012-0105-z.


\bibitem{bumu6} U. Buijs, J. Guti\'errez and A. Murillo, Derivations, the Lawrence-Sullivan interval
and the Fiorenza-Manetti mapping cone,  {\em Math. Zeitschrift}, DOI: 10.1007/s00209-012-1040-x.

\bibitem{bumu2} U. Buijs and A. Murillo, Basic constructions in rational homotopy theory of function spaces, {\em Annales de l'Inst. Fourier}, {\bf56}(3) (2006), 815--838.


\bibitem{bumu} U. Buijs and A. Murillo, The Lawrence-Sullivan construction is the right model of $I^+$, to appear in \textit{Alg. and Geom. Top.}






\bibitem{ke} A. Cattaneo, B. Keller, C. Torossian and A. Brugui\`eres
D\'eformation, Quantification, Th\'eorie de Lie, {\em Panoramas et Synthèses-Parutions}, {\bf 20} (2005).

\bibitem{la} J. Chuang and A. Lazarev,  L-infinity maps and twistings, {\em Homology,
Homotopy and Applications}, {\bf 13}(2)  (2011), 175--195.


\bibitem{fehatho} Y. F\'elix, S. Halperin and J.-C. Thomas, Rational
homotopy theory, {\em Springer GTM }, {\bf 205 } (2000).

\bibitem{femutan} Y. F\'elix, D. Tanr\'e and A. Murillo, Fibrewise rational homotopy theory, {\em Journal of Topology} {\bf 3}(4) (2010), 743--758.

\bibitem{gers} M. Gerstenhaber, On the deformation of rings and algebras, {\em Ann. of Math. (2)} {\bf 79} (1964), 59--103.

\bibitem{getz} E. Getzler, Lie theory for nilpotent $L_\infty$-algebras. \textit{Ann. of Math. (2)} {\bf 170}(1) (2009),  271--301.


\bibitem{gulambs} V.K.A.M. Gugenheim, L. Lambda and J. Stasheff, Perturbation theory in differential homological algebra II, \textit{Illinois J. Math.} \textbf{35}(3) (1991), 357--373.

    \bibitem{haef} A. Haefliger, Rational Homotopy of the space of
sections of a nilpotent bundle, {\em Trans. Amer. Math. Soc.},
{\bf 273} (1982), 609--620.

\bibitem{hi} V. Hinich, Homological algebra of homotopy algebras, {\em Comm. in Algebra}, {\bf 25}(10) (1997), 3291--3323.

\bibitem{hi2} V. Hinich, DG Coalgebras and formal stacks, {\em Journal of Pure and Applied Algebra}, {\bf 162}(2--3) (2001), 209--250.


\bibitem{huebschka} J. Huebschmann and T. Kadeishvili, Small models for chain algebras, \textit{Math. Z.}, \textbf{207}(2) (1991),  245--280.





\bibitem{kon} M. Kontsevich, Deformation quantization of Poisson manifolds, \textit{Lett. Math. Phys.}, \textbf{66}(3) (2003), 157--216.

\bibitem{kontsoi} M. Kontsevich and Y. Soibelman, Deformations of algebras over operads and Deligne's conjecture. {\em G. Dito and D. Sternheimer (eds) Conf\'erence Mosh\'e Flato 1999 , Vol. I (Dijon 1999), Kluwer
Acad. Publ., Dordrecht} (2000) 255--307.

\bibitem{lannes} J. Lannes, Sur la cohomologie modulo p des p-groupes ab\'eliens \'el\'ementaires, Proc.
Durham Symp. on Homotopy Theory 1985, {\em LMS Lecture Note Ser.}, {\bf 117}, (1987), 97--116.

\bibitem{lasu} R. Lawrence and D. Sullivan, A free diferential Lie algebra for the interval, \textit{Preprint} arXiv:math/0610949.




\bibitem{lova} J. L. Loday \and B. Vallette, Algebraic Operads, \textit{Preprint available at} math.unice.fr/$\sim$brunov/Operads.pdf.

\bibitem{ma} M. Majewski, Rational homotopical models and
uniqueness, {\em Mem. Amer. Math. Soc.}, {\bf 682} (2000).

\bibitem{mane} M. Manetti, Deformation theory via differential graded Lie algebras, Seminari di Geometria Algebrica 1998-1999, Scuola Normale Superiore (1999).

\bibitem{patan} P. E. Parent and D. Tanr\'e, Lawrence-Sullivan models for the interval, {\em Topology and its Appl.}, {\bf 159} (2012), 371--378.


\bibitem{qui2} D. Quillen,  Homotopical Algebra, {\em Springer GTM }, {\bf 43 } (1967).

\bibitem{qui} D. Quillen, Rational homotopy theory, {\em Ann.
of Math.}, {\bf 90} (1969), 205--295.


\bibitem{tradzeisu} T. Tradler, M. Zeinalian, with an appendix by D. Sullivan, Infinity structure of Poincar\'e duality spaces, {\em Algebraic and Geometric Topology}, {\bf 7} (2007), 233--260.



\end{thebibliography}
\end{document}